\documentclass[12pt]{amsproc}
\usepackage{graphicx,xcolor} 
\usepackage{amsthm}
\usepackage{amssymb}
\usepackage{amsmath}
\usepackage{mathrsfs}
\usepackage{amsfonts}
\usepackage{geometry}
\usepackage{enumerate}

\newtheorem{theorem}{Theorem}[section]
\newtheorem{corollary}[theorem]{Corollary}
\newtheorem{lemma}[theorem]{Lemma}

\theoremstyle{definition}
\newtheorem{example}[theorem]{Example}
\newtheorem{remark}[theorem]{Remark}
\newtheorem{problem}[theorem]{Problem}

\newcommand{\GL}{\mathrm{GL}}
\newcommand{\AGL}{\mathrm{AGL}}
\newcommand{\PSL}{\mathrm{PSL}}
\newcommand{\PSiL}{\mathrm{P\Sigma L}}
\newcommand{\SL}{\mathrm{SL}}
\newcommand{\PGL}{\mathrm{PGL}}
\newcommand{\PG}{\mathrm{PG}}

\newcommand{\Sp}{\mathrm{Sp}}
\newcommand{\Suz}{\mathrm{Suz}}
\newcommand{\soc}{\mathrm{soc}}
\newcommand{\PSU}{\mathrm{PSU}}
\newcommand{\Ree}{\mathrm{Ree}}
\newcommand{\Sym}{\mathrm{Sym}}

\newcommand{\mat}{\mathrm{M}}
\newcommand{\M}{\mathcal{M}}
\newcommand{\U}{\mathcal{U}}
\newcommand{\HS}{\mathrm{HS}}
\newcommand{\Co}{\mathrm{Co}}
\newcommand{\ch}{\mathrm{char}}
\DeclareMathOperator{\lindim}{LinDim}
\newcommand{\F}{\mathbb F}

\title{Linear dimension of group actions}
\author{Alice Devillers}
\address[Devillers, Giudici, Morgan]{Department of Mathematics and Statistics, The University of Western Australia, Perth WA 6009, Australia}
\email{alice.devillers@uwa.edu.au, michael.giudici@uwa.edu.au}
\email{luke.morgan@uwa.edu.au}
\author{Michael Giudici}
\author{Daniel R. Hawtin}
\address[Hawtin]{Faculty of Mathematics, University of Rijeka, Rijeka 51000, Croatia}
\email{dhawtin@math.uniri.hr}
\author{Lukas Klawuhn}
\address[Klawuhn]{Department of Mathematics, Paderborn University, Warburger Str.\ 100, 33098 Paderborn, Germany.}
\email{klawuhn@math.upb.de}
\author{Luke Morgan}

\begin{document}

\begin{abstract}
Two fundamental ways to represent a group are as permutations and as matrices. 
In this paper, we study linear representations of  groups that intertwine with a permutation representation. 
Recently, D'Alconzo and Di Scala \cite{DAlconzo2024} investigated how small the matrices in such a linear representation can be. The minimal dimension of such a representation is the \emph{linear dimension of the group action} and this has applications in cryptography and cryptosystems.

We develop the idea of linear dimension from an algebraic point of view by using the theory of permutation modules. We give structural results about representations of minimal dimension and investigate the implications of faithfulness, transitivity and primitivity on the linear dimension. Furthermore, we compute the linear dimension of several classes of finite primitive permutation groups. We also study wreath products, allowing us to determine the linear dimension of imprimitive group actions. Finally, we give the linear dimension of almost simple finite $2$-transitive groups, some of which 
may be used for further applications in cryptography. Our results also open up many new questions  about   linear representations of group actions.
\end{abstract}

 \maketitle

\section{Introduction}
 D'Alconzo and Di Scala \cite{DAlconzo2024} recently introduced the concept of a group action representation and the linear dimension of a group action. This is motivated by an application to cryptography and to determine how susceptible to attack cryptosystems based on group actions are.  In this paper, we develop the general theory of group action representations and  compute the linear dimensions of some well-known group actions.

 Let $G$ be a group acting on a set $\Omega$, let $F$ be a field and $V$ be a vector space over~$F$. A pair $(\rho,\varphi)$ is a \emph{representation of the group action over $F$} if $\rho:G\rightarrow \GL(V)$ is a homomorphism and $\varphi:\Omega\rightarrow V$ is injective such that 
\begin{equation}\label{intertwin}
(\omega^g)\varphi= \left(\omega\varphi\right)^{g\rho}
\end{equation}
for all $\omega\in\Omega$ and $g\in G$. We refer to (\ref{intertwin}) as the \emph{intertwining property}. We say that $V$ \emph{affords} the representation $(\rho,\varphi)$. Note that we follow the notation of Wielandt \cite{Wielandt} and so $\omega^g$ denotes the image of $\omega$ under the element $g$. We also write our functions on the right. 

Given a group $G$ acting on sets $\Omega_1$ and $\Omega_2$, we say that the two actions are \emph{equivalent} if there is a 
bijection $\tau:\Omega_1\rightarrow \Omega_2$ such that $(\omega^g)\tau=\left(\omega \tau\right)^g$. Note that if we have a homomorphism $\rho:G\rightarrow\GL(V)$ for some vector space $V$, then we obtain a natural action of $G$ on $V$ by $v^g:=v^{g\rho}$.  This will give rise to a representation of the group action of $G$ on $\Omega$ over $F$ precisely when we can find a set $X$ of vectors in $V$ such that the action of $G$ on $\Omega$ is equivalent to the action of $G$ on $X$. Note that we will often write $v^g$ instead of $v^{g\rho}$ when the context is clear.

We define the \emph{$F$-linear dimension} of the action of $G$ on $\Omega$ to be 
\begin{align*}
    \lindim_F(G,\Omega) = & \min  \left \{\  \dim(V) {\quad  \left   | \quad \begin{array}{c}
     V  \textrm{ affords a representation} \\ \textrm{ of the action of }G  \text{ on } \Omega \text{ over }F \end{array} \right . } \right \}.
\end{align*}
 When $|F|=q$ (a prime power), we denote $\lindim_F(G,\Omega)$ by $\lindim_q(G,\Omega)$ and this is called the \emph{$q$-linear dimension} by D'Alconzo and Di Scala. If $V$ affords a representation of the action of $G$ on $\Omega$ over $F$ and $\dim(V)=\lindim_F(G,\Omega)$, then we call $V$ a \emph{witness} of $\lindim_F(G,\Omega)$. Note that witnesses need not be unique, see Example \ref{eg:twowitnesses} where we give two witnesses of $\lindim_2(C_6,6)=4$.

 D'Alconzo and Di Scala \cite{DAlconzo2024} determined $\lindim_q(C_n,\{1,2,\ldots,n\})$, for $C_n$ a cyclic group of order $n$, and showed that for $n>2$ and any prime power $q$ we have $\lindim_q(S_n,\{1,2,\ldots,n\})=n-1$. Note that this is not always the smallest dimension of a faithful irreducible representation of $S_n$, which is $n-2$ when the characteristic of the field divides $n$. Moreover, the witness is not irreducible, see Section \ref{sec:witness} for more details.

Given an action of a group $G$ on a set $\Omega$, the \emph{permutation module}  over $F$ is the vector space $F^\Omega$ with basis $\{e_\omega\mid \omega\in\Omega\}$ such that $(e_\omega)^g=e_{\omega^g}$. Note that $F^\Omega$ affords a representation of the action of $G$ on $\Omega$, where $\varphi$ maps each $\omega\in\Omega$ to the basis vector $e_\omega \in F^\Omega$. Thus, as observed by D'Alconzo and Di Scala \cite[Remark 10]{DAlconzo2024}, we have that $\lindim_F(G,\Omega)\leqslant |\Omega|$. 

One of our first important observations is that a witness of $\lindim_F(G,\Omega)$ is a quotient of $F^\Omega$ (Corollary \ref{cor:isquo}) and this leads to the following theorem. The action of $G$ on $\Omega$ is called \emph{primitive} if the only partitions of $\Omega$ preserved by $G$ are the trivial partitions, namely, the one with only one part, and the one where all parts are singletons.

\begin{theorem}\label{thm:prim}
Let $G$ act primitively on $\Omega$ such that the induced permutation group $G^{\Omega}$ is not cyclic. Then $\lindim_F(G,\Omega)$ is the smallest codimension of a submodule of the permutation module~$F^\Omega$ that is not a hyperplane. Moreover, if a witness for $\lindim_F(G,\Omega)$ is reducible then it has a unique non-zero proper submodule $W$ and $W$ has codimension one.
\end{theorem}

Theorem \ref{thm:prim} follows from Corollaries \ref{cor:smallestcodim} and \ref{cor:almost irred}. We also observe that Theoreom \ref{thm:prim} does not hold for imprimitive groups as demonstrated by Examples \ref{eg:S4on6} and \ref{eg:cp wr sd}.

Primitive groups arise as the natural class of actions to study. We see in Lemma~\ref{lem:intransLWRBound} that the linear dimension of an intransitive action can be bounded above and below in terms of the linear dimensions of the actions on each orbit. Moreover, we give a sufficient condition for when this upper bound is sharp (Theorem \ref{thm:intrans=}). Thus, it is natural to reduce the study of the linear dimension of a group action to the transitive case. Moreover, an imprimitive permutation group naturally embeds in a wreath product of a primitive permutation group with a transitive group (as we outline in Section \ref{sec:imprim}) and we are able to determine the linear dimension of  wreath products in their imprimitive actions. Note that if $H\leqslant G$, then $\lindim_F(H,\Omega)\leqslant \lindim_F(G,\Omega)$ (Lemma~\ref{lem: h leq g}).

\begin{theorem}
\label{thm:imprimitive theorem}
Let $K\leqslant \Sym(\Delta)$ be   primitive, let $L \leqslant \Sym(\Sigma)$, and let $G= K \wr L$ acting imprimitively on the set $\Omega=\Delta \times \Sigma$. Let $k=|\Delta|$ and $\ell=|\Sigma|$. Then
$$\lindim_F(G, \Omega) = \begin{cases}
     \ell+1 & \text{if }  $K$ \text{ is regular and }k =\mathrm{char}(F), \\
    \ell \lindim_F(K,\Delta) & \text{otherwise. }
\end{cases}$$
\end{theorem}
    
Finite primitive permutation groups are described by the O'Nan-Scott Theorem, see for example \cite[Section 4]{cameron1999permutation}. One major class of primitive groups consists of wreath products in product action. For $K\leqslant \Sym(\Delta)$ and $L\leqslant S_\ell$, the wreath product $K\wr L$ has a natural action on the Cartesian power $\Delta^\ell$ (see Section \ref{sec:prodaction} for more details). Lemma \ref{lem: upper bound for product action} gives the natural upper bound $\lindim_F(K\wr L,\Delta^\ell)\leqslant \ell\lindim_F(K,\Delta)$ and we are able to give the exact values when $K=\Sym(\Delta)$ and $L=S_\ell$.

\begin{theorem}\label{thm:prodaction}
Let $\Delta$ be a finite set of cardinality $k\geqslant 3$ and let $\ell\geqslant 2$. Suppose that $\ch(F)=p\geqslant 0$.
Then $$\lindim_F(S_k\wr S_\ell,\Delta^\ell)=
\begin{cases} 
(k-2)\ell+1 &\text{if }p\mid k, \\
(k-1)\ell & \text{if }p \nmid k.
\end{cases}$$
\end{theorem}
 
The finite primitive groups that are not contained in a wreath product in product action are called \emph{basic} and are
of affine, diagonal or almost simple type \cite[Theorem 4.6]{cameron1999permutation}. 
We consider the largest possible affine groups in Example \ref{eg:agldp}. Almost simple groups have many primitive actions. One of our main results determines the linear dimension for the action of $A_n$ or $S_n$ on the set $\Omega_k$ of all $k$-subsets of $\{1,2,\ldots,n\}$. 

 \begin{theorem}
 \label{thm:Snksets}
    Let $G=A_n$ or $S_n$ with $n\geqslant 10$ and let $F$ be a field of characteristic~$p\geqslant 0$.  Let $1\leqslant k< \frac{n}{2}$. Then
    \[
    \lindim_F(G,\Omega_k) =    \begin{cases} n-2 &\textrm{ if }p\mid n \textrm{ and } p\mid k,  \\
  n-1 &\text{ otherwise.}\\
\end{cases}
    \]
\end{theorem}

Finally, we determine the linear dimensions of the almost simple 2-transitive groups.
 
\begin{theorem}\label{thm:2transLindim}
 Let $F$ be a field of characteristic $p\geqslant 0$ and let $G$ be an almost simple $2$-transitive group of degree $n$ acting on $\Omega$. Then either $\lindim_F(G,\Omega)=n-1$, or $\lindim_F(G,\Omega)$ is as in Table~\ref{tab:2transLinDim}. 
\end{theorem}

We pose further natural problems in Section \ref{sec:problems}.

\begin{table}[ht]
 \begin{center}
 \begin{tabular}{c|cccc}
  \hline
   & $\soc(G)$ & $n$ & $\lindim_F(G,\Omega)$ & conditions\\
  \hline
  1 & $\PSL_d(q)$ & $\frac{q^d-1}{q-1}$ & $\binom{d+p-2}{d-1}^t$ & $d\geqslant 3$, $q=p^t$ \\
  
  2 & ${\rm A}_7$ & $15$ & $4$ & $p=2$ \\
  
  3 & $\Sp_{2d}(2)$ & $2^{2d-1}-2^{d-1}$ & $2d+1$ & $p=2$, $d\geqslant 3$ \\
  4 & $\Sp_{2d}(2)$ & $2^{2d-1}+2^{d-1}$ & $2d+1$ & $p=2$, $d\geqslant 2$ \\
  
  5 & $\PSL_2(q)$ & $q+1$ & $\frac{q+1}{2}$ & $q\equiv\pm 1\pmod 8$ \\  
   &  & &  &  $p=2$, $G\leqslant\PSiL_2(q)$ \\
  
  6 & $\PSL_2(q)$ & $q+1$ & $\frac{q+1}{2}$ & $q\equiv\pm 3\pmod 8$ \\
   & & & & $\F_4\leqslant F$, $G\leqslant\PSiL_2(q)$ \\
  
  7 & $\Suz(q)$ & $q^2+1$ & $m(q-1)/2+1$ & $m^2=2q$, $p\mid q+1+m$ \\
  
  8 & $\PSU_3(q)$ & $q^3+1$ & $q^2-q+1$ & $p\mid q+1$ \\
  
  9 & $\Ree(q)$ & $q^3+1$  & $q^2-q+1$ & $p=2$ \\
  10 & $\Ree(q)$ & $q^3+1$ & $\frac{mq^2-m}{6}+\frac{q^2-q}{2}+1$ & $m^2=3q$, $2\neq p\mid q+1$ \\
  11 & $\Ree(q)$ & $q^3+1$ & $\frac{m}{3}(q^2-1)+1$ & $m^2=3q$, $2\neq p\mid q+m+1$ \\

  12 & $\mat_{22}$ & $22$ & $10$ & $p=2$ \\
  13 & $\mat_{23}$ & $23$ & $11$ & $p=2$ \\
  14 & $\mat_{24}$ & $24$ & $12$ & $p=2$ \\
  15 & $\mat_{11}$ & $12$ & $6$ & $p=3$ \\
  
  16 & $\PSL_2(11)$ & $11$ & $5$ & $p=3$ \\
  
  17 & $\HS$ & $176$ & $21$ & $p=2$ \\
  18 & $\HS$ & $176$ & $49$ & $p=3$ \\
  
  19 & $\Co_3$ & $276$ & $23$ & $p=2$ \\
  20 & $\Co_3$ & $276$ & $126$ & $p=3$ \\
  \hline
 \end{tabular}
 \caption{Almost simple $2$-transitive groups $G$ of degree $n$, listed by their socle, for which $\lindim_F(G,\Omega)\neq n-1$, where $\ch(F)=p$.}
 \label{tab:2transLinDim}
 \end{center}
\end{table}

The motivation for the definition of linear dimension \cite{DAlconzo2024} was in proving that certain group actions are not appropriate for proposed cryptographic applications. Roughly speaking, Proposition~14 and Corollary~15 of \cite{DAlconzo2024} imply that if $G_\lambda\leqslant\Sym(\Omega_\lambda)$ is some family of permutation groups, indexed by $\lambda\in\mathbb{N}$, and there exists a finite field $\F_q$ such that $\lindim_q(G_\lambda,\Omega_\lambda)$ is logarithmic in $|G_\lambda|$ (or $|\Omega_\lambda|$), then the family of actions of $G_\lambda$ on $\Omega_\lambda$ does not satisfy the desirable \emph{multiple one-wayness} assumption, nor the \emph{weak unpredictability} and \emph{weak pseudorandomness} assumptions defined in \cite{alamati2020cryptographic}. In particular, taking $\lambda$ to be the \emph{security parameter}, as defined in \cite[Section~2.2]{DAlconzo2024}, gives $\log(|G_\lambda|)=O(\mathsf{poly}(\lambda))$ and $\log(|\Omega_\lambda|)=O(\mathsf{poly}(\lambda))$, and \cite[Proposition~14]{DAlconzo2024} shows that if $\lindim_F(G_\lambda,\Omega_\lambda)=\mathsf{poly}(\lambda)$ for some finite field~$F$, then the group action does not satisfy the multiple one-wayness assumption, provided a certain additional condition holds. 

Given the discussion in the preceding paragraph, Theorem~\ref{thm:2transLindim} shows, for example, that the $2$-transitive actions of $A_n$, $ S_n$ or $\Sp_{2d}(2)$ are not likely to satisfy the multiple one-wayness assumption. On the other hand, if we let $q$ be a function exponential in $\lambda$, and let $G_\lambda$ be one of the families $\PSL_2(q)$, $\Suz(q)$, $\PSU_3(q)$ or $\Ree(q)$ acting $2$-transitively on the relevant set $\Omega_\lambda$, then there is a polynomial relationship between $|G_\lambda|$ and $\lindim_F(G_\lambda,\Omega_\lambda)$ for any field $F$. Hence, \cite[Proposition~14]{DAlconzo2024} does not apply to these four families of groups, leaving them open as potential candidates for cryptographic applications.

\subsubsection*{Acknowledgements} \quad We thank Gerhard Hiss and Peter Sin for their very useful comments regarding the modular representations of finite $2$-transitive groups.

We thank the Centre for the Mathematics of Symmetry and Computation at The University of Western Australia for organising the annual research retreat of 2025 at which this research project began.

The research of Michael Giudici and Luke Morgan was  supported by  the Australian Research Council via the Discovery Project grant DP230101268.

Daniel Hawtin was partially funded by the Croatian Science Foundation under the project number HRZZ--IP--2022--10--4571 and is grateful for the support of the Centre for the Mathematics of Symmetry and Computation's 2024 Cheryl E.~Praeger Visiting Fellowship.

 Lukas Klawuhn was partially funded by the Deutsche Forschungsgemeinschaft (DFG, German Research Foundation) – Project-ID 491392403 – TRR 358. He would also like to thank the DAAD (German Academic Exchange Service) for funding a 2-month research visit to The University of Western Australia during which this project started.

\section{Representation theory preliminaries}

In this section, we outline some representation theory concepts that will be needed later. 

Let $V$ be a vector space over a field $F$ and $G$ be a group. We call $V$ an $FG$-module if we can define $v^g\in V$ for all $v\in V$ and $g\in G$ such that $(v+w)^g=v^g+w^g$, $v^{gh}=(v^g)^h$, $(\lambda v)^g=\lambda (v^g)$ and $v^{1_G}=v$ for all $v,w\in V$ and $g,h\in G$. Note that every $g \in G$ then induces an $F$-linear map on $V$ and so we obtain a homomorphism $\rho:G\rightarrow \GL(V)$. Conversely, given a representation $\rho:G\rightarrow \GL_n(F)$ of $G$, that is, $\rho$ is a homomorphism, then $V=F^n$ inherits an $FG$-module structure via $v^g=v^{g\rho}$ for all $v\in V$ and $g\in G$.  We will interchangeably move between the language of $FG$-modules and representations. 

The \emph{character} of a representation $\rho$ is the map $\chi:G\rightarrow F$ such that $\chi(g)$ is the trace of the matrix $(g)\rho$.  The character of an $FG$-module is the character of the corresponding representation. When $F=\mathbb{C}$, we call a character of $G$ an \emph{ordinary} character. Moreover, any $\mathbb{C}G$-module can be decomposed as a direct sum of irreducible $\mathbb{C}G$-modules and the corresponding character $\chi$ can be written as a sum of irreducible characters. The irreducible characters in the sum are called the \emph{irreducible constituents} of $\chi$.

Let $G$ be a group acting on the set $\Omega$, let $F$ be a field and let $F^\Omega$ be the permutation module over $F$ with basis $\{e_\omega\mid \omega\in\Omega\}$. The character $\pi$ of the permutation module is called the \emph{permutation character} and for each $g\in G$ we have that $\pi(g)$ is the number of elements of $\Omega$ that are fixed by~$g$. Now $F^\Omega$ is equipped with the symmetric bilinear form 
\begin{equation}\label{F^n:bilinear form}   
\left\langle \sum \lambda_\omega e_\omega,\sum \mu_\omega  e_\omega \right \rangle=\sum \lambda_\omega\mu_\omega.
\end{equation}
 Let $C=\langle \sum e_\omega\rangle$
be the subspace of ``constant'' vectors of $F^\Omega$. With
 the non-degenerate bilinear form defined in \eqref{F^n:bilinear form}, we have that $C^\perp=\{\sum \lambda_\omega e_\omega\mid \sum\lambda_\omega=0\}$, which has dimension $n-1$. Note that both $C$ and $C^\perp$ are submodules of the permutation module, and we  refer to $C^\perp/(C \cap C^\perp)$ as the \emph{fully deleted permutation module}.

We have the following result about the smallest faithful irreducible representations of $A_n$.

\begin{theorem}
\label{thm:mods for Sn}
Let $G=A_n$  with $n\geqslant 5$  and let $F$ be a field.  If $  U  $ is a proper non-zero submodule of the permutation module $F^n$, then $U=C$ or $U=C^\perp$. Moreover, the following hold:
\begin{enumerate}[$(1)$]
    \item if $\ch(F)=0$ or $\ch(F)\nmid n$, then $F^n \cong C \oplus C^\perp$;
        \item if $\ch(F)\mid n$, then $ C \subseteq C^\perp  $;
        \item for $n\geqslant 10$ the fully deleted permutation module $C^\perp  / (C^\perp \cap C)$ is the unique faithful irreducible $FG$-module of smallest dimension. Furthermore, this  dimension equals $n-1$ when
        $\ch(F)=0$ or $\ch(F)\nmid n$, and equals $n-2$ when $\ch(F)\mid n$.
\end{enumerate}
\end{theorem}
\begin{proof}
Parts (1) and (2) are elementary. For part (3)
see \cite[Proposition 5.3.5]{KL}. Note that~\cite{Wagner1}~missed that $A_9$ has three irreducible representations of dimension 9, as corrected by James~\cite[p421]{james}.
\end{proof}

Let $G$ be a group with subgroup $H$. If $V$ is an $FH$-module, then we denote the \emph{induced $FG$-module} by $V\uparrow_H^G$.
This is a vector space of dimension $|G:H|\dim(V)$. See \cite[Chapter 5]{Isaacs} for more details. If $\chi$ is the character of $V$, then the character of $V\uparrow_H^G$ is denoted by $\chi \uparrow_H ^G$. If $G$ is a transitive permutation group with point stabiliser $H$, then the permutation character of $G$ is equal to $1\uparrow_H^G$ by \cite[(5.14) Lemma]{Isaacs}.

\section{Group action representation preliminaries}

Let $G$ be a group acting on a set $\Omega$ and let $(\rho,\varphi)$ be a representation of the group action over a field $F$. We define $G^{\Omega}$ to be the permutation group induced by $G$ on~$\Omega$.

The bound below was observed in the proof of \cite[Proposition 24]{DAlconzo2024} for the particular case $H=A_n$ and $G=S_n$, we give a general statement.

\begin{lemma}
\label{lem: h leq g}
    Let $G$ act on $\Omega$ and let $H \leqslant G$. Then $\lindim_F(H,\Omega) \leqslant \lindim_F(G,\Omega)$.
\end{lemma}
\begin{proof}
    Let $V$ be a witness to $\lindim_F(G,\Omega)$ afforded by $(\rho,\varphi)$. Then $(\rho|_H,\varphi)$ is a representation of the group action of $H$ on $\Omega$ over $F$. Thus 
    $$\lindim_F(H,\Omega) \leqslant \dim(V) = \lindim_F(G,\Omega)$$
    as required.
\end{proof}

The next result is concerned with the effect of field extensions on the linear dimension. If $F$ is an extension field of $L$ and $V$ is an $LG$-module, we obtain an $FG$-module $V^F$ by `extending' scalars: this module is simply $V\otimes_L F$.

\begin{lemma}
\label{lem: ext field}
Let $G$ act on $\Omega$, and let $F$ and $L$ be fields such that $|F:L|=\ell$. Then \[ \lindim_F(G,\Omega) \leqslant\lindim_L(G,\Omega)\leqslant \ell \lindim_F(G,\Omega).\]
\end{lemma}
\begin{proof}
Let $W$ be a witness for $\lindim_F(G,\Omega)$ afforded by $(\rho,\varphi)$. Then $W$ is also an $L$-vector space of dimension $\ell \dim(W)$ and elements of $(G)\rho$ also induce $L$-linear transformations.  Thus, $W$ also affords a representation for the action of $G$ on $\Omega$ over $L$ and the right hand inequality of the statement holds.

 Now let $V$ be a witness to $\lindim_L(G,\Omega)$ and consider $V^F$ as an $FG$-module.  We have  an embedding $\GL(V)\leqslant \GL(V^F)$ such that $(v\otimes \lambda)^g=v^g\otimes \lambda$  for all $g\in\GL(V)$, $v\in V$ and $\lambda\in F$.  Further, if $V$ affords the representation $(\rho,\varphi)$, then $V\otimes F$ affords the representation $(\rho',\varphi')$ where~$\rho'$ is induced by the embedding $\GL(V)\leqslant \GL(V^F)$ and $\varphi':\omega \mapsto \omega\varphi \otimes 1$. Thus,
 \[
    \lindim_F(G,\Omega)\leqslant \dim_F(V^F) = \dim_L(V) = \lindim_L(G,\Omega). \qedhere
\]
\end{proof}

 We make heavy use of the fact that, given a representation $(\rho,\varphi)$ for the action of $G$ on $\Omega$, the actions of $G$ on $\Omega$ and on  $(\Omega)\varphi$ (via $\rho$) are equivalent. In fact, the converse to this statement also holds, which allows us to give upper bounds on the $F$-linear dimension of a group action.

\begin{lemma}
\label{lem: rep affords action}
    Let $G$ act transitively on a set $\Omega$ and $\rho : G \rightarrow \GL(V)$ be a  faithful  $F$-linear representation. Then there exists an injection $\varphi : \Omega \rightarrow V$ such that $(\rho,\varphi)$ is a representation of the group action over $F$ if and only if there exist $\alpha\in\Omega$ and $v\in V$ such that  $(G\rho)_v=(G_\alpha)\rho$. 
\end{lemma}
\begin{proof}
    In the forward direction, let  $\varphi : \Omega \rightarrow V$  be as in the statement. Pick  $\alpha \in \Omega$ and set $v=\alpha\varphi$. For $g\in G$ we have  $v^{g\rho} = (\alpha\varphi)^{g\rho}=(\alpha^g)\varphi $, so $g\rho \in (G\rho)_v$ if and only if $g\in G_\alpha$. That is, $(G_\alpha)\rho = (G\rho)_v$. 

    Conversely, if there is $\alpha\in\Omega$ and $v\in V$ such that $(G\rho)_v = (G_\alpha)\rho$, then the action of $G\rho$ on the orbit $v^{G\rho}$ is equivalent to the action of $G\rho$ on the set of cosets $[G\rho :(G \rho)_v]=[G\rho : (G_\alpha)\rho]=:\Omega'$. Since $\rho$ is faithful, this action is equivalent to the action of $G$ on the set of right cosets $[G:G_\alpha]$, which is  equivalent to the action of $G$ on $\Omega$. Hence, there is a bijection $\varphi : \Omega \rightarrow v^{G\rho}$ which we may view as an injection $\varphi : \Omega \rightarrow V$. Since the actions on the sets $\Omega$ and $v^{G\rho}$ are  equivalent, this means $(\rho,\varphi)$ is a representation of the group action over $F$.
\end{proof}

We will also make use of the fact that the permutation group induced on $(\Omega)\varphi$ is a quotient of~$(G)\rho$.

\begin{lemma}\label{lem:cyclic}
Let $G$ be a finite group acting on a set $\Omega$ such that $\lindim_F(G,\Omega)=1$. Then $G^{\Omega}$ is cyclic. 
\end{lemma}
\begin{proof}
Suppose that $\lindim_F(G,\Omega)=1$. Then there exists a homomorphism $\rho:G\rightarrow\GL_1(F)$ and an injection $\varphi : \Omega \rightarrow F$. Since finite subgroups of the multiplicative group of a field are cyclic, it follows that $(G)\rho$ is cyclic. Moreover, since the actions of $G$ on $\Omega$ and $(\Omega)\varphi$ are permutationally isomorphic, they induce isomorphic permutation groups. Thus, $G^{\Omega}$ is isomorphic to a quotient of $(G)\rho$ and so $G^{\Omega}$ is cyclic.   
\end{proof}

Next, we make some observations about the faithfulness of the group action and the linear representation.
Let $G$ act on $\Omega$ and let $(\rho,\varphi)$ be a representation of the group action over a field $F$ where $\rho: G \to \GL(V)$. Let $K$ be the  kernel of the action of $G$ on $\Omega$.  Since $\ker(\rho)$ acts trivially on $V$, each $g\in \ker(\rho)$ fixes each $\omega\varphi$. Hence, $g$ fixes each $\omega \in \Omega$, 
which implies that
$g\in K$, and 
thus
$\ker(\rho) \leqslant K$. In particular, faithful permutation actions must have faithful linear representations.  

As a partial converse, if $V$ is spanned by $\{ \omega\varphi \mid \omega \in \Omega \}$ (which happens if $V$ is a witness of $\lindim_F(G,\Omega)$), then $K$ acts trivially on $V$, and so $K \leqslant \ker(\rho)$. In particular, if $V$ is a witness to $\lindim_F(G,\Omega)$, then $K=\ker(\rho)$. Note that this is not true for an arbitrary representation as the following example shows.

\begin{example}
Let $G=\SL(2,3)$ and let $\Omega$ be the set of $1$-subspaces of $\F_3^2$. Then $|\Omega|=4$, and $Z:=Z(G)\cong C_2$ is the kernel of the action of $G$ on $\Omega$. Moreover, $G/Z = \PSL(2,3)\cong A_4$. 
Since~$G$ is not cyclic, and $G/Z$ is not a subgroup of $\GL(2,3)$, we have $\lindim_3(G,\Omega)=3$, with witness $V=\mathbb F_3^\Omega/\langle e_1+e_2+e_3 +e_4\rangle$ affording the representation $(\rho,\varphi)$. As described above, $\ker(\rho)=Z$.

Let $W=\mathbb F_3^2$ be the natural module for $G$. Now $V\oplus W$ is an $FG$-module and we abuse notation to use $\varphi$ to refer to the embedding of $\Omega$ into $V\oplus W$. Let $\rho'$ be the linear representation of $G$ on $V\oplus W$. Then $(\rho',\varphi)$ is a representation of the group action over $F$, and $\ker(\rho')=1$ since $G$ acts faithfully on $W$. In particular, the kernel of the action $Z$ does \emph{not} equal $\ker(\rho')$.
\hfill $\diamond$
\end{example}

We have the following lemma which means that we   only need to consider faithful actions to determine the linear dimension.

\begin{lemma}
Let $G$ act on the set $\Omega$ with kernel $K$ and let $F$ be a field. Then $\lindim_F(G,\Omega) = \lindim_F(G/K,\Omega)$.
\end{lemma}
\begin{proof}
 Let $V$ be a witness to $\lindim_F(G,\Omega)$ afforded by $(\rho,\varphi)$. As mentioned above, $K=\ker(\rho)$, so $(\rho,\varphi)$ is also a  representation of the group action of $G/K$ on $\Omega$ over $F$. This proves $\lindim_F(G/K,\Omega)\leqslant \lindim_F(G,\Omega)$.
 
 Conversely, if $(\rho',\varphi')$ is a representation of the group action of $G/K$ on $\Omega$, then by composing the canonical map $\pi :G \rightarrow G/K$ with $\rho'$ we obtain a representation $(\pi\rho',\varphi')$ for the group action of $G$ on $\Omega$ over $F$. Taking $(\rho',\varphi')$ to be a representation affording a witness, we find that $\lindim_F(G,\Omega)\leqslant \lindim_F(G/K,\Omega)$, and we are done.
\end{proof}

Given an action of a group $G$ on a set $\Omega$, a \emph{$G$-congruence} is an equivalence relation $\sim$ on $\Omega$ such that $\alpha\sim\beta$ if and only if $\alpha^g\sim\beta^g$ for all $\alpha,\beta\in\Omega$ and $g\in G$. A transitive action is called \emph{primitive} if the only $G$-congruences on $\Omega$ are the \emph{trivial} ones -- equality and the universal relation.  Otherwise, it is called \emph{imprimitive} and a partition of $\Omega$ given by the set of equivalence classes of a non-trivial $G$-congruence is called a \emph{system of imprimitivity}.

The linearity of the action of $\GL(V)$ on $V$ also gives rise to a natural $G$-congruence.

\begin{lemma}
Suppose that the action of $G$ on $\Omega$ has a representation $(\rho,\varphi)$ over $F$ on $V$. For all $\alpha,\beta\in\Omega$ define $\alpha\sim\beta$ if and only if $\alpha\varphi=\lambda(\beta\varphi)$ for some $\lambda\in F\backslash\{0\}$. Then $\sim$ is a $G$-congruence.
\end{lemma}
\begin{proof}
Note that being a scalar multiple is an equivalence relation on $V$ that is preserved by $\GL(V)$, and hence by $(G)\rho$. Since $\varphi$ is a bijection, $\sim$ is also an equivalence relation on $\Omega$. 

Suppose that $\alpha\sim \beta$. Then there exists $\lambda\in F\backslash\{0\}$ such that $\alpha\varphi=\lambda(\beta\varphi)$. Then 
\[
    (\alpha^g)\varphi=(\alpha\varphi)^{g\rho}=(\lambda(\beta\varphi))^{g\rho}=\lambda (\beta\varphi)^{g\rho}=\lambda ((\beta^g)\varphi).
\]
Hence, $\alpha^g\sim\beta^g$ and so $\sim$ is a $G$-congruence.    
\end{proof}

\begin{corollary}
 Suppose that the action of $G$ on $\Omega$ has a representation $(\rho,\varphi)$ over $F$ on $V$ and that $G$ acts primitively on $\Omega$. Then  either $\dim(V)=1$ or distinct images $(\omega)\varphi$ lie in distinct 1-dimensional subspaces. 
\end{corollary}

\section{Quotients and a witness}
\label{sec:witness}

The following theorem underpins many of our results and all of our computational statements. In fact, it is a rather innocent exercise from the text of Aschbacher \cite[Exercise 6(2), pg.51]{Aschbacher}.

\begin{theorem}
\label{thm:fund thm}
 Let $V$ afford a representation $(\rho,\varphi)$ of the action of $G$ on $\Omega$ over~$F$.
Suppose that $$V=\langle \omega\varphi\mid \omega\in\Omega\rangle.$$
Then there is a surjective $G$-module homomorphism $\pi : F^\Omega \rightarrow V$ such that $\omega \varphi = e_\omega \pi$. In particular, $V$ is a quotient of the permutation module $F^\Omega$.    
\end{theorem}
\begin{proof}
 Let $\{ e_\omega \mid \omega \in \Omega\}$ denote the standard basis of $ F^{\Omega}$. We define $\pi : F^\Omega \rightarrow V$ by $\pi : e_\omega \mapsto \omega\varphi$ and extend to $F^\Omega$ by linearity. The fact that $\{\omega\varphi \mid \omega \in \Omega\}$ spans $V$  means that $\pi$ is a surjection, hence $V \cong F^\Omega/\ker(\pi)$. Since $(\rho,\varphi)$ is a representation of the group action, $\pi$ is a $G$-module homomorphism.
 \end{proof}

\begin{corollary}\label{cor:isquo}
Let $V$ be a witness to $\lindim_F(G,\Omega)$.  Then $V$ is a quotient of the permutation module $F^\Omega$.
\end{corollary}
\begin{proof}
 Note that if $V$ is afforded by $(\rho,\varphi)$, then  $W = \langle \omega\varphi \mid \omega \in \Omega \rangle$ is an invariant subspace of $V$. By the minimality of $V$, we have $V=W$ and the result follows from Theorem~\ref{thm:fund thm}.
 \end{proof}

We also obtain the following identification of the permutation module.
 \begin{corollary}\label{cor:idpermmodule}
 Let $V$ afford a representation $(\rho,\varphi)$ of the action of $G$ on $\Omega$ over~$F$.
Suppose that $\{\omega\varphi\mid \omega\in\Omega\}$ is a basis for $V$. Then $V$ is isomorphic to the permutation module $F^\Omega$.    
 \end{corollary}

By Corollary~\ref{cor:isquo}, the linear dimension is always realised on a quotient of the permutation module $F^\Omega$.  
We observe that it is possible that none of the submodules of dimension $\lindim_F(G,\Omega)$ afford a representation of the group action over $F$ of $G$, as the following example shows.

\begin{example}\label{ex:submodule}
    We know that $\lindim_F(S_n,\{1,2,\ldots,n\}) = n-1$ if $n>2$. Let $C$ be the submodule of $F^n$ of constant vectors so that $C^\perp$ is the zero-sum submodule of $F^n$, which has dimension $n-1$, and define
    \[
        \varphi: \{1,\ldots,n\} \to C^\perp,\;i \mapsto ne_i -(e_1+\cdots+e_n).
    \]
    If $\ch(F)$ does not divide $n$, then $\varphi$ is injective and has the intertwining property. So $\lindim_F(S_n,n)$ is realised on $C^\perp$, a submodule of $F^n$.

    If $\ch(F)$ divides $n$, then there is no injective map $\varphi:\{1,\ldots,n\} \to C^\perp$ with the intertwining property.
    Indeed, assume that $\varphi:\{1,\ldots,n\} \to C^\perp$ were such a map and fix $i$. Then $i\varphi=\sum_{j=1}^n \lambda_j e_j$ with $\sum \lambda_j=0$.
    Using the intertwining property with the transposition $g=(k,\ell)$ where $i,k,\ell$ are distinct, we get that 
    $$i\varphi=i^g\varphi=(i\varphi)^{g\rho}=\sum_{\substack{j=1\\j\neq k,\ell}}^n \lambda_j e_j+\lambda_\ell e_k+\lambda_k e_\ell.$$
    Thus, $\lambda_k=\lambda_\ell$ for any $k,\ell\neq i$. Say $\lambda_k=\mu$ for all $k\neq i$. Therefore $$i\varphi=\sum_{\substack{j=1\\j\neq i}}^n\mu e_j-(n-1)\mu e_i=\mu(e_1+\cdots+e_n)$$ 
    since $\ch(F)$ divides $n$. Thus, $\varphi$ is not injective.
      Hence, as $C^\perp$ is the only submodule of dimension $n-1$ of the permutation module $F^n$, the linear dimension is not realised on a submodule in this case. \hfill $\diamond$
\end{example}

As in the example above, analysing the the action of a transposition was a crucial tool
in the proof of D'Alconzo and Di Scala that $\lindim_F(S_n,\Omega)=n-1$ \cite{DAlconzo2024} and will also be used several times in the rest of the paper.

If $V$ is a witness to $\lindim_F(G,\Omega)$, then $V=F^\Omega/U$ by Corollary~\ref{cor:isquo} for some submodule $U$ of $F^\Omega$ of codimension $\lindim_F(G,\Omega)$. If $U$ has a complement, say $F^\Omega = U \oplus W$, then $W$ has dimension $\lindim_F(G,\Omega)$ and $W$ affords a representation of the group action as $F^\Omega / U \cong W$. The submodule $C$ in Example~\ref{ex:submodule} does not have a complement if $\ch(F) \mid n$.

 Corollary \ref{cor:isquo} raises the question of which quotients of the permutation module~$F^\Omega$ afford a representation of the action. We develop some general results about quotients.

 Let $V$ be a vector space over the field $F$ that affords a representation $(\rho,\varphi)$ of the action of $G$ on $\Omega$. Let $U$ be a subspace of $V$ that is invariant under $(G)\rho$. Then the map 
 \begin{equation}
 \label{eq: barrho}
     \overline{\rho}:G\rightarrow \GL(V/U) \quad \text{defined by} \quad g\overline{\rho}\, :\, U+v\, \mapsto\, U + v^{g\rho}
 \end{equation}
 is a homomorphism.
 Further, we define 
 \begin{equation}
 \label{eq: barvarphi}
 \overline{\varphi}:\Omega\rightarrow V/U \quad \text{by} \quad  \omega\overline{\varphi}=U+\omega\varphi
 \end{equation}
 and note that for all $g\in G$ we have that 
 $$(\omega^g)\overline{\varphi}=U+(\omega^g)\varphi=U+(\omega\varphi)^{g\rho}=(U+\omega\varphi)^{g\overline{\rho}}=(\omega\overline{\varphi})^{g\overline{\rho}}.$$
  Thus, the intertwining property holds for $(\overline{\rho},\overline{\varphi})$.

 To determine whether or not $(\overline{\rho},\overline{\varphi})$ is a representation of the action of $G$ on $\Omega$, we need to determine if $\overline{\varphi}$ is an injection. One instance where this fails is   when $\omega\varphi\in U$ for some $\omega\in\Omega$ with $|\omega^G|>1$. Since $U$ is $G$-invariant, we will then have that $\alpha\overline\varphi=U$ for all $\alpha\in\omega^G$.

\begin{theorem}\label{thm:quotient}
Let $V$ be a vector space over $F$ that affords the representation $(\rho,\varphi)$ of $G$ on $\Omega$. Let $U$ be a subspace of $V$ that is invariant under $(G)\rho$. For $\alpha,\beta\in\Omega$ define $\alpha\sim\beta$ if and only if $U+\alpha\varphi=U+\beta\varphi$. Then $\sim$ is a $G$-congruence.    
\end{theorem}
\begin{proof}
Clearly $\sim$ is an equivalence relation. Let $g\in G$. Since $U$ is $g\rho$-invariant, we have that $(U+\alpha\varphi)^{g\rho}=U+(\alpha\varphi)^{g\rho}=U+(\alpha^g)\varphi$.
Thus, if $\alpha\sim \beta$, then $U+(\alpha^g)\varphi=(U+\alpha\varphi)^{g\rho}=(U+\beta\varphi)^{g\rho}=U+(\beta^g)\varphi$ and so $\alpha^g\sim\beta^g$.    
\end{proof}

\begin{corollary}\label{cor:prim}
 Let $V$ be a vector space over $F$ that affords the representation $(\rho,\varphi)$ of $G$ on $\Omega$.
Suppose that $G$ acts primitively on $\Omega$ and $G^\Omega$ is not cyclic. Suppose further that $V=\langle\omega\varphi\mid \omega\in\Omega\rangle$ and $U$ is a $(G)\rho$-invariant subspace of codimension greater than one. Then the map $\overline{\varphi}:\Omega\rightarrow V/U$ defined in \eqref{eq: barvarphi}
is an injection.    
\end{corollary}
\begin{proof}
Since $G$ is primitive, the only $G$-congruences are equality and the universal relation. If the relation $\sim$ as defined in Theorem \ref{thm:quotient} is the universal relation, then $U+\alpha\varphi=U+\beta\varphi$ for all $\alpha,\beta\in\Omega$. Since $U$ is a proper $(G)\rho$-invariant subspace and the set $(\Omega)\varphi$ spans $V$, we deduce that $V/U=\langle U+\omega\varphi\rangle$ for some $\omega\in\Omega$. Thus $\lindim_F(G,\Omega)=1$, which by Lemma \ref{lem:cyclic} contradicts $G^\Omega$ not being cyclic. Thus $\sim$ is equality and $\overline{\varphi}$ is an injection.
\end{proof}

We can now give a partial converse of Theorem \ref{thm:fund thm}.

\begin{theorem}\label{thm:quopermmodule}
Suppose that $G$ acts primitively on $\Omega$ and $G^\Omega$ is not cyclic, and let $U$ be a submodule of the permutation module $F^\Omega$ of codimension greater than one. Then $F^\Omega/U$ affords a representation of the action of $G$ on $\Omega$. 
\end{theorem}

\begin{proof}
Define $\varphi:\Omega\rightarrow F^\Omega$ by $\omega\varphi=e_\omega$ and let $\rho:G\rightarrow \GL(F^\Omega)$ be the natural homomorphism induced by the action of $G$ on $F^\Omega$. Then $\langle \omega\varphi\mid\omega\in\Omega\rangle=F^\Omega$. Thus, Corollary \ref{cor:prim} implies that $(\overline{\rho},\overline{\varphi})$ is a representation for $G$ on $\Omega$ over $F$.
\end{proof}

This gives the following, perhaps surprising, corollary.
\begin{corollary}
Let $G$ be a primitive permutation group on $\Omega$ that is not cyclic and let $U$ be a submodule of the permutation module $F^\Omega$ of codimension at least 2. Then $G$ acts faithfully on~$F^\Omega/U$.
\end{corollary}

\begin{proof}
 By Theorem \ref{thm:quopermmodule}, $F^\Omega$ affords a representation of the action of $G$ on $\Omega$ and so the kernel of the induced homomorphism $\overline{\rho}:G\rightarrow \GL(F^\Omega/U)$ is contained in the kernel of the action of $G$ on $\Omega$. Since $G$ is a permutation group, it acts faithfully on $\Omega$ and so $\overline{\rho}$ is faithful.
\end{proof}

We also obtain the following corollaries about a witness to $\lindim_F(G,\Omega)$.

\begin{corollary}\label{cor:quotientPermMod}
Suppose that $G$ acts primitively on $\Omega$ and $G^\Omega$ is not cyclic. Suppose that $V$ is a witness of $\lindim_F(G,\Omega)$. Then $V$ is isomorphic to $F^\Omega/U$ as $G$-modules, where $U$ is a submodule of the permutation module $F^\Omega$ of largest possible dimension that is not a hyperplane.
\end{corollary}

\begin{corollary}\label{cor:smallestcodim}
Suppose that $G$ acts primitively on $\Omega$ and $G^\Omega$ is not cyclic. Then $\lindim_F(G,\Omega)$ is the smallest codimension of a submodule of the permutation module $F^\Omega$ that is not a hyperplane.
\end{corollary}

Since the permutation module is self-dual, it follows that if $W$ is a submodule of $F^\Omega$ of dimension $d$, then $W^\perp$ has codimension $d$ and $V/W^\perp\cong W^*$, the dual of $W$. 
So in the primitive case, Corollary \ref{cor:smallestcodim} implies that it is enough to know the smallest dimension of a submodule of $F^\Omega$ that is not 1-dimensional (even though it is possible that no such submodule affords a representation of the group action).

\begin{corollary}
\label{cor:almost irred}
Suppose that $G$ acts primitively on $\Omega$, that $G^\Omega$ is nonabelian and let $V$ be a witness to $\lindim_F(G,\Omega)$. Suppose that $V$ is not irreducible. Then $V$ has a unique nonzero proper submodule $W$ and $W$ has codimension one. In particular, $W$ is irreducible and $V$ cannot be expressed as a direct sum of two proper non-zero submodules.
\end{corollary}
\begin{proof}
Let $W$ be a proper non-zero submodule of $V$.    By Corollary~\ref{cor:quotientPermMod} and the correspondence theorem, $V=F^\Omega/U$ and  $W=U_0/U$ for  submodules $U$ and $U_0$ of $F^\Omega$ with $U < U_0 < F^\Omega$ and $U$ maximal with respect to not being a hyperplane. By the previous corollary, $U_0$ must be a hyperplane of $F^\Omega$. Hence, $W$ has codimension $1$ in $V$.

Suppose now that $W'$ is a proper non-zero submodule of $V$ distinct from $W$. Then $W'=U_1/U$ for some submodule $U_1$ of $F^\Omega$ with $U <U_1$ and $U_1 \neq U_0$. By the previous paragraph, $U_1$ is also a hyperplane, and $U < U_0 \cap U_1$, so by the maximality of $U$ we have $U=U_0 \cap U_1$. This means $U$ has codimension $2$, that is, $\dim(V)=2$ and $V=W\oplus W'$. Since $G$ preserves the two submodules $W$ and $W'$, we have that $G^\Omega \leqslant \GL(W) \times \GL(W')$. Since  $\dim(W)=\dim(W')=1$, this implies $G$ is abelian, a contradiction. Hence, $W$ is the unique proper non-zero submodule of $V$, and moreover,~$W$ is irreducible.
\end{proof}

\begin{lemma}
\label{lem: kernel on soc(V)}
Suppose that $G$ acts primitively on $\Omega$, that $G^\Omega$ is nonabelian and let $V=F^\Omega/U$ be a witness to $\lindim_F(G,\Omega)$, with $p=\ch(F)$. Suppose that $V$ is not irreducible and let $W$ be the unique  nonzero proper submodule of $V$. Let $K$ be the kernel of the action of $G$ on $U$. Then $G$ acts trivially on $V/W$ and if $p$ is a prime, then $K$ is a $p$-group.
\end{lemma}

\begin{proof}
Write $W=U_1/U$ so that $U < U_1$. Note that $U_1$ has codimension $1$ in $F^\Omega$. If $U_1 \neq C^\perp$, then $C^\perp \cap U_1$  has codimension $2$ in $F^\Omega$. It follows that $F^\Omega/(U_1 \cap C^\perp)$ is a witness, but this is a contradiction to the previous corollary as $U_1/(U_1\cap C^\perp)$ and $C^\perp/(U_1\cap C^\perp)$ would be distinct submodules. Thus, $U_1 = C^\perp$ and we see that  $G$ acts trivially on 
$$F^\Omega/C^\perp = F^\Omega/U_1 \cong (F^\Omega/U)/(U_1/U) = V/W.$$

Now let $K$ be the kernel of the action of $G$ on $W$. Then $K$ acts trivially on $W$ and on $V/W$. Hence, $K$ is a $p$-group.
\end{proof}

The following example shows that Corollary \ref{cor:smallestcodim} does not necessarily hold for imprimitive permutation groups.

\begin{example}\label{eg:S4on6}
Let $G=S_4$ act on the set $\Omega$ of right cosets of $\langle(1,2,3,4)\rangle$ in $G$. Then $|\Omega|=6$. Since $G^\Omega$ is not cyclic, we have by Lemma \ref{lem:cyclic} that $\lindim_3(G,\Omega)\neq 1$. Since the action of $G$ on $\Omega$ is faithful and $\GL_2(3)$ does not have a subgroup isomorphic to $S_4$, it follows that $\lindim_3(G,\Omega)\geqslant 3$. By a \textsc{Magma} \cite{magma} calculation, we see that the permutation module $V=\F_3^\Omega$ has proper  non-zero submodules with codimensions 1, 2, 3, 4 and 5. Since $\lindim_3(G,\Omega)\geqslant 3$, we see that $\lindim_3(G,\Omega)$ is not the smallest codimension of a submodule that is not a hyperplane.

Similarly, if $V=\F_2^\Omega$, we see that $V$ has submodules of codimensions 1, 2, 3, 4 and~5. However, $\lindim_2(G,\Omega)=4$.\hfill $\diamond$
\end{example}

We also have a variation on Theorem \ref{thm:quotient} for particular submodules as follows.
Let $V$ be a vector space that affords a representation $(\rho,\varphi)$ of the action of $G$ on $\Omega$. 
We define 
$$C_V(G)=\{v\in V\mid v^{g\rho}=v\textrm{ for all }g\in G\}$$
and note that $C_V(G)$ is invariant under $(G)\rho$.

\begin{lemma}
    \label{lem:fix pt space zero}
Suppose that $G$ acts on a set $\Omega$ and $G_\alpha \neq G_\beta$ for all distinct points $\alpha,\beta \in \Omega$. If~$V$ is a witness to $\lindim_F(G,\Omega)$, then $C_V(G)=0$.
\end{lemma}
\begin{proof}
Let $(\rho,\varphi)$ be a representation afforded by $V$. Define $\overline{\rho}:G\rightarrow \GL(V/C_V(G))$ and $\overline{\varphi}:\Omega\rightarrow V/C_V(G)$ as in \eqref{eq: barrho} and \eqref{eq: barvarphi} with $U=C_V(G)$. Suppose that  $\alpha\overline{\varphi}=\beta\overline{\varphi}$. Then there exists $v\in C_V(G)$ such that $\alpha\varphi=v+\beta\varphi$. Let $g\in G$. Then 
 $$(\alpha^g)\varphi=(\alpha\varphi)^{g\rho}=(v+\beta\varphi)^{g\rho}=v+ (\beta\varphi)^{g\rho}=v+(\beta^g)\varphi$$
 and so for $g\in G_\alpha$ we have
 $$v+(\beta^g)\varphi = (\alpha^g)\varphi=\alpha\varphi = v+\beta\varphi$$
 which implies $\beta\varphi = (\beta^g)\varphi$.  Since $\varphi$ is an injection, it follows that $g\in G_\beta$ and so $G_\alpha \leqslant G_\beta$. A symmetrical argument yields $G_\beta \leqslant G_\alpha$, and so we have equality. This contradicts the assumption that $G_\alpha\neq G_\beta$ and so $\overline{\varphi}$ is an injection. Moreover, $(\overline{\rho},\overline{\varphi})$ is a representation for the action of $G$ on $\Omega$. Since $V$ is a witness to $\lindim_F(G,\Omega)$, it follows that $C_V(G)=0$.
 \end{proof}

\begin{remark}
\label{rem: primitive implies pt stabs distinct}
    Note that equality of point stabilisers defines a $G$-congruence on $\Omega$, and so primitive groups with $G_\alpha\neq 1$ for all $\alpha\in\Omega$ (i.e.~primitive groups that are not regular) satisfy the hypothesis of Lemma~\ref{lem:fix pt space zero}.
\end{remark} 

To see that the conditions in Lemma \ref{lem:fix pt space zero} are necessary, take $G=C_2$ acting on itself by right multiplication. Then $\lindim_{2}(G,\Omega)=2$ and a witness is $V=\F_2^2 = \langle e_1,e_2\rangle$. We have that $C_V(G)=\{0,e_1+e_2\}$ so that $\dim C_V(G)=1$.

For another example, let $G=S_4$ acting on the set $\Omega$ of right cosets of $H = \langle(1,2,3,4)\rangle$. Let $\alpha = H$ and $\beta = H(1,4)(2,3)$, so that $\alpha \neq \beta$ but $G_\alpha = G_\beta = H$. As seen in Example \ref{eg:S4on6},  $\lindim_2(G,\Omega)=4$. Moreover, for a witness $V$, we see by a {\sc Magma} computation that  $\dim C_V(G) = 1$.

\section{Intransitive actions}

We begin with the following lemma.

\begin{lemma}\label{lem:intransLWRBound}
Let $\Omega=O_1\cup O_2\cup\ldots\cup O_t$ be a disjoint union and suppose that $G$ acts on $\Omega$ with each $O_i$ being $G$-invariant.
 Then $$\max\{\lindim_F(G,O_i)\mid 1 \leqslant i \leqslant t \}\leqslant \lindim_F(G,\Omega)\leqslant \sum_{i=1}^t\lindim_F(G,O_i)$$
    \end{lemma}
\begin{proof}
 Let $W$ be a witness to $\lindim_F(G,\Omega)$ afforded by $(\rho,\varphi)$. For each $i$, we have that $(\rho,\varphi_{\mid O_i})$ is a representation of the action of $G$ on  $O_i$. Thus 
$$\dim W \geqslant \dim \langle (O_i)\varphi\rangle\geqslant \lindim_F(G,O_i)$$
and so we have that $\max\{\lindim_F(G,O_i)\mid i\in\{1,\ldots, t\}\}\leqslant \lindim_F(G,\Omega)$.

For each $i$, suppose that $V_i$ affords a  representation $(\rho_i,\varphi_i)$ of the action of $G$ on $O_i$ over $F$.
Now, let $V=V_1\oplus V_2\oplus\cdots\oplus V_t$ be the formal direct sum of the $V_i$ and define $\rho:G\rightarrow \GL(V)$ by $g\rho= (g\rho_1,g\rho_2,\ldots,g\rho_t)$. Then $\rho$ is a homomorphism. Let $\varphi:\Omega\rightarrow V$ be the map such that if $\omega\in\Omega$ lies in the orbit $O_i$ then $\omega\varphi=\omega\varphi_i$. Then $\varphi$ is an injection. Moreover, for all $\omega\in\Omega$ and $g\in G$, if $\omega\in O_i$ then $\omega^g\in O_i$ and we have that 
$$(\omega^g)\varphi=(\omega^g)\varphi_i=(\omega\varphi_i)^{g\rho_i}=(\omega\varphi)^{g\rho},$$
 since $g\rho_{\mid V_i}=g\rho_i$. Thus $(\rho,\varphi)$ is a representation of the action of $G$ on $\Omega$, so $\dim W\leqslant \dim V$  and the result follows.
\end{proof}

Note that Lemma \ref{lem:intransLWRBound} applies when the $O_i$ are orbits of $G$. Example \ref{ex:52} shows the lower bound in Lemma \ref{lem:intransLWRBound} is sharp.

\begin{example}\label{ex:52}
Suppose that $G$ acts faithfully on each orbit $O_i$ and that the actions of $G$ on orbits are pairwise equivalent, that is, for each $i$ and $j$ there is a bijection $\chi_{i,j}:O_i\rightarrow O_j$ such that $(\omega^g)\chi_{i,j}=(\omega\chi_{i,j})^g$ for all $\omega\in O_i$ and $g\in G$. Thus there is a vector space $V$ that is a witness  for each $\lindim_F(G,O_i)$ and afforded by $(\rho,\varphi_i)$.

Define $\psi_{i,j}:(O_i)\varphi_i\rightarrow (O_j)\varphi_j$ by $\psi_{i,j}=\varphi_i^{-1}\chi_{i,j}\varphi_j$ and note that $\psi_{i,j}$ is a bijection. Then for all $\omega\in O_i$ and $g\in G$ we have that 
\begin{align*}
    ((\omega\varphi_i)^{g\rho})\psi_{i,j}&=(\omega^g)\varphi_i\psi_{i,j}\\
    &=(\omega^g)\chi_{i,j}\varphi_j\\
    &=((\omega\chi_{i,j})^g)\varphi_j\\
    &=((\omega\chi_{i,j})\varphi_j)^{g\rho}\\
    &=((\omega\varphi_i\varphi_i^{-1})\chi_{i,j}\varphi_j)^{g\rho}\\
    &=((\omega\varphi_i)\psi_{i,j})^{g\rho}
\end{align*}
and so the actions of $(G)\rho$ on the $(O_i)\varphi_i$ are pairwise equivalent.

 Let $F$ be a field such that  
 $|F|>t$.
Then $\lindim_F(G,\Omega)=\lindim_F(G,O_1)$. We see this as follows. Let $\lambda_2,\ldots,\lambda_t$ be distinct elements of $F\backslash\{0,1\}$. Then define $\varphi:\Omega\rightarrow V$ by 
$$(\omega)\varphi=\left\{\begin{array}{ll}
     (\omega)\varphi_1 & \textrm{if $\omega\in O_1$},  \\
      \lambda_i((\omega)\chi_{i,1}\varphi_1) &  \textrm{if $\omega\in O_i$ for $i>1$.}
    \end{array}\right.$$
    Since the $\lambda_i$ are distinct,  $\varphi$ is an injection.
To check that $(\rho,\varphi)$ is a representation of $G$ on $\Omega$, we just need to check the intertwining property for all $\omega\in O_2\cup\ldots\cup O_t$. Suppose that $\omega\in O_i$ for some $i>1$ and let $g\in G$. Recall the maps $\psi_{i,1}$ defined above.
Then 
\begin{align*}
 (\omega^g)\varphi & = \lambda_i((\omega^g)\chi_{i,1}\varphi_1) \\
 &=\lambda_i(((\omega\chi_{i,1})^g)\varphi_1)\\
&=\lambda_i(\omega\chi_{i,1}\varphi_1)^{g\rho}\\
&=(\lambda_i((\omega)\chi_{i,1}\varphi_1))^{g\rho}\\
&=((\omega)\varphi)^{g\rho}.
\end{align*}
Hence, $(\rho,\varphi)$ is a representation of the action of $G$ on $\Omega$ and so $\lindim_F(G,\Omega)=\lindim_F(G,O_1)$. \hfill $\diamond$
 \end{example}

The next example shows that we can still get  a sharp lower bound
when the actions of $G$ on its orbits are not equivalent.

\begin{example}
Let $G=S_n$ with $n\geqslant 5$ and let $\Omega$ be the set of all non-empty subsets of $\{1,2,\ldots,n\}$ of size less than $n/2$. 
For each $i<n/2$, let $O_i$ be the set of $i$-subsets. Then the orbits of $G$ on $\Omega$ are the $O_i$ and the action of $G$ on $O_1$ is equivalent to the natural action of $S_n$ on $\{1,2,\ldots,n\}$. 

Let $F^n$ be the permutation module for $G$ acting on $\{1,2,\ldots, n\}$. Let $C$ be the subspace of $F^n$ consisting of all constant vectors and let $V=F^n/C$. As we will see  in Lemma~\ref{lem: upper bound for S_n on k sets}, $\lindim(S_n,O_i)\leqslant n-1$ as $V$ affords a representation of $G$ on each $O_i$ where the images of each element of $\Omega$ are distinct. Hence, $V$ also affords a representation of the action of $G$ on $\Omega.$
 Thus, $\lindim_F(G,\Omega)\leqslant n-1$. Since $\lindim(S_n,O_1)=\lindim_F(G,\{1,2,\ldots,n\})=n-1$, 
 we have that $\lindim(S_n,\Omega)\geqslant \lindim(S_n,O_1)=n-1$ by Lemma~\ref{lem:intransLWRBound}. Thus, $\lindim(S_n,\Omega)= 
 n-1=\max\{\lindim_F(G,O_i)\mid 1\leqslant i\leqslant \lfloor\frac{n-1}{2}\rfloor\}$. \hfill $\diamond$
\end{example}

We now give an example with equality for the upper bound in Lemma \ref{lem:intransLWRBound}.
 The next example shows that the upper bound in Lemma \ref{lem:intransLWRBound} is also sharp.

\begin{example}\label{eg:intransSymO_i}
Let $G=\Sym(O_1)\times \Sym(O_2)\times \cdots\times \Sym(O_t)$ act intransitively on $\Omega$, the disjoint union of $O_1,O_2,\ldots,O_t$ with $|O_i|\geqslant 3$ for all $i$.
We claim that $\lindim_F(G,\Omega)=\sum_i \lindim_F(G,O_i)$.
Note that $$\lindim_F(G,O_i)=\lindim_F(\Sym(O_i),O_i)=|O_i|-1.$$ Thus
$\sum_i \lindim_F(G,O_i)=|\Omega|-t$.

Let $V$ be a vector space that affords a representation $(\rho,\varphi)$ of the action of $G$ on $\Omega$ over $F$. Write $v_\omega=\omega\varphi$ for each $\omega\in\Omega$ and note that $v_{\omega^g}=(v_\omega)^{g\rho}$ for all $g\in G$. For each $i$, fix one element $\alpha_i\in O_i$, and let $\Delta=\Omega\setminus\{\alpha_1,\alpha_2,\ldots,\alpha_t\}$. Note that $|\Delta|=|\Omega|-t.$

Consider the set $\{v_\omega|\omega\in\Delta\}$ and suppose that $\sum_{\omega\in\Delta}\lambda_\omega v_\omega=0$ for some $\lambda_\omega\in F$. Let $\delta\in \Delta$. Then $\delta\in O_i\setminus\{\alpha_i\}$ for exactly one $i$. Consider the transposition $g=(\delta,\alpha_i)\in G$.
Then 
$$0=0^{g\rho}=\sum_{\omega\in\Delta}\lambda_\omega (v_\omega)^{g\rho}=\sum_{\omega\in\Delta}\lambda_\omega v_{\omega^g}=\lambda_{\delta} v_{\alpha_i}+\sum_{\omega\in\Delta\setminus\{\delta\}}\lambda_\omega v_\omega.$$
It follows that $\lambda_{\delta} v_{\alpha_i}=\lambda_{\delta} v_{\delta}$. Since $\varphi$ is injective, we have $v_{\alpha_i}\neq v_{\delta}$, and hence $\lambda_{\delta}=0$.
Therefore the set $\{v_\omega|\omega\in\Delta\}$ is linearly independent and  so $\dim(V)\geqslant |\Delta|=|\Omega|-t.$ 
Now the claim follows from Lemma \ref{lem:intransLWRBound}. \hfill $\diamond$
\end{example}

We can generalise Example \ref{eg:intransSymO_i}.

\begin{theorem}
\label{thm:intrans=}
For $i=1,\ldots,t$, let $H_i \leqslant \Sym(O_i)$ be permutation groups such that for each $i$ we have $(H_i)_\alpha\neq (H_i)_\beta$ for distinct points $\alpha,\beta\in O_i$.
Let $G=H_1\times H_2\times\cdots\times H_t$ act intransitively on the disjoint union $\Omega=O_1\cup O_2\cup \cdots \cup O_t$. 
Then $\lindim_F(G,\Omega)=\sum_{i=1}^t \lindim_F(H_i,O_i)$.
\end{theorem}
\begin{proof}
Let $V$ be a witness to $\lindim_F(G,\Omega)$ such that $V$ affords the representation $(\rho,\varphi)$.
Let $\alpha,\beta$ be distinct points of $\Omega$. If there exists $1\leqslant i \leqslant t$ such that $\alpha,\beta \in O_i$, then 
$$G_\alpha = (H_i)_\alpha \times \prod_{j\neq i } H_j \quad  \text{and} \quad G_\beta = (H_i)_\beta \times \prod _{j\neq i} H_j$$ whence $G_\alpha \neq G_\beta$ by the hypothesis of the theorem. If $\alpha \in O_i$ and $\beta \in O_j$ for distinct $i,j$, then $G_\alpha = (H_i)_\alpha \times \prod_{k\neq i} H_k$ and $G_\beta = (H_j)_\beta \times \prod_{k\neq j} H_k$, whence $G_\alpha \neq G_\beta$. Thus, by Lemma~\ref{lem:fix pt space zero} we have $C_V(G) = 0$.

For each $i$, let $V_i=\langle \omega\varphi\mid \omega\in O_i\rangle$. Then $(\rho_{\mid H_i},\varphi_{\mid O_i})$ is a representation for the action of $H_i$ on $O_i$. Thus, $\dim V_i\geqslant \lindim_F(H_i,O_i)$.

For each $j\neq i$ we have that $H_j$ acts trivially on $O_i$ and so $(H_j)\rho$ acts trivially on $V_i$. Similarly, $(H_i)\rho$ acts trivially on $\sum_{j\neq i} V_j$.  It follows that $G=H_1\times \cdots \times H_t$ acts trivially on $V_i\cap \sum_{j\neq i} V_j$, and hence $V_i \cap \sum_{j \neq i} V_j \leqslant C_{V}(G) = 0$. This implies $V= V_i \oplus (\sum_{j \neq i} V_j)$, and repeating this argument gives $V = \bigoplus_{i=1}^tV_i$. Hence,
\[
    \dim V = \sum_{i=1}^t \dim V_i \geqslant \sum_{i=1}^t \lindim_F(H_i,O_i),
\]
and combined with Lemma~\ref{lem:intransLWRBound} we have equality as in the statement.
\end{proof}

The next example shows that $\lindim_F(G,\Omega)$ can also reach values strictly between the two bounds.

\begin{example}\label{eg:intransSymO_i2}
Let $1<s< t$, and let $G=\Sym(O_1)\times \Sym(O_2)\times \cdots\times \Sym(O_s)$ act intransitively on $\Omega$, the disjoint union of $O_1,O_2,\ldots,O_t$ with $|O_i|\geqslant 3$ for all $i$ and $O_j$ is a copy of $O_s$ for all $j>s$, so that $\Sym(O_s)$ has equivalent actions on all $O_j$ for $j\geqslant s$.
More precisely, the action is defined by 
$$\omega^{(\sigma_1,\sigma_2,\ldots,\sigma_s)}=
\left\{\begin{array}{ll}
     \omega^{\sigma_i}& \textrm{if $\omega\in O_i$ for $i\leqslant s$}  \\
     \omega^{\sigma_s}&  \textrm{if $\omega\in O_i$ for $i> s$.}
    \end{array}\right.$$
We claim that if $|F|> t-s+1$, then $\lindim_F(G,\Omega)=\sum_{i=1}^s \lindim_F(G,O_i)$.
Note that $$\lindim_F(G,O_i)=\lindim_F(\Sym(O_i),O_i)=|O_i|-1.$$ Thus
$\sum_{i=1}^s \lindim_F(G,O_i)=\left(\sum_{i=1}^s|O_i|\right)-s$.

Let $\Delta_1=O_1\cup O_2\cup \ldots \cup O_{s-1}$ and $\Delta_2=O_s\cup O_{s+1}\cup \ldots\cup O_t$. As seen in Example \ref{eg:intransSymO_i},
 $\lindim_F(G,\Delta_1)=\sum_{i=1}^{s-1}\lindim_F(G,O_i)$, while in Example \ref{ex:52} we saw that $\lindim_F(G,\Delta_2)=\lindim_F(G,O_s)$ if $|F|>t-s+1$. Thus, by Lemma \ref{lem:intransLWRBound} we deduce that
\[
    \lindim_F(G,\Omega)\leqslant\lindim_F(G,\Delta_1)+\lindim_F(G,\Delta_2)=\sum_{i=1}^{s}\lindim_F(G,O_i).
\]
Alternatively, let $\Delta_3=O_1\cup O_2\cup \ldots\cup O_s$.
 By Lemma \ref{lem:intransLWRBound}, we have that
 \[
    \lindim_F(G,\Delta_3)\leqslant \lindim_F(G,\Omega)
\]
and as in Example \ref{eg:intransSymO_i} we see that
\[
    \lindim_F(G,\Delta_3)=\sum_{i=1}^{s}\lindim_F(G,O_i).
\]
Thus, $\lindim_F(G,\Omega)=\sum_{i=1}^{s}\lindim_F(G,O_i)$, as claimed. \hfill $\diamond$
\end{example}

\section{Imprimitive actions}\label{sec:imprim}

Let $K\leqslant\Sym(\Delta)$ and $ L\leqslant \Sym(\Sigma)$ with $\Sigma=\{1,2,\ldots,\ell\}$. We define the \emph{wreath product} $K\wr  L:=K^\ell\rtimes  L$, such that for all $(h_1,h_2,\ldots, h_\ell)\in K^\ell$ and $\sigma\in  L$ we have that $(h_1,h_2,\ldots,h_\ell)^\sigma=(h_{1\sigma^{-1}},h_{2\sigma^{-1}},\ldots,h_{\ell\sigma^{-1}})$. Then $K\wr  L$ acts on $\Delta\times \Sigma$ by $(\delta,i)^{(h_1,h_2,\ldots,h_\ell)\sigma}=(\delta^{h_i},i^\sigma)$. This action is imprimitive with system of imprimitivity $\{ \{(\delta,i)\mid \delta\in\Delta\}\mid i\in \Sigma\}$. 

We begin with a fundamental example of an imprimitive wreath product that, as a linear group, stabilises a subspace of dimension $1$.

\begin{example}
\label{eg:cp wr sd}
    Suppose that $F = \F_p$ and let $V$ be an $F$-vector space of dimension $\ell+1$ with standard basis $\{e_1,\ldots,e_{\ell+1}\}$.   Let $S \cong S_\ell$ be the subgroup of $\GL(V)$ that fixes $e_{\ell+1}$ and naturally permutes the first $\ell$ basis vectors $e_1,\ldots,e_\ell$ by acting on subscripts. For $1\leqslant i \leqslant \ell$, define $y_i \in \GL(V)$ to be the map that fixes  $e_{j}$ for $j\neq i$ and maps $e_i$ to $e_i+e_{\ell+1}$. Let $P=\langle y_i \mid 1\leqslant i \leqslant \ell\rangle \cong C_p^\ell$ and  set $G=\langle P, S\rangle.$ By direct calculation, we see $G \cong C_p \wr S_\ell$ and we set $\Omega = \{1,\ldots,p\}\times\{1,\ldots,\ell\}$ so that $G$ acts  imprimitively on $\Omega$ as explained above. We compute that
    $$(e_1)^G = \{ e_i + \lambda e_{\ell+1} \mid \lambda \in F\}$$
    so that $|(e_1)^G|=p\ell$.  Furthermore, we see that $G$ preserves the partition of $(e_1)^G$ into $\ell$ parts of size $p$ given by  $B_i = \{e_i+\lambda e_{\ell+1}\mid \lambda \in F \}$.  Thus, $V$ is a representation of the imprimitive action of $G=C_p\wr S_\ell$ on $p\ell$ points over $F$ and we have
    \begin{equation}
    \label{eq: cp wr sd d+1}
        \lindim_p(C_p \wr S_\ell, \Omega) \leqslant \ell+1.
    \end{equation}
    
    Set $U=\langle e_{\ell+1} \rangle$ and observe that $U$ is fixed by $G$ and so $G$ acts on $V/U$. We find that $P$ acts trivially on $V/U$ since 
    \[
        y_i : e_i + U \mapsto e_i + e_{\ell+1}+U = e_i+U,
    \]
    and thus $G/P=S_\ell$ acts on $V/U = \langle e_i +U \rangle $. Hence, $V/U$, as an $FS_\ell$-module, is the permutation module for $S_\ell$ over $F$. From the decomposition of the permutation module $F^\ell$, we see that $V/U$ has  $2$ or $3$ composition factors according to $p\nmid \ell$ or $p\mid \ell$, respectively. In Theorem~\ref{thm:imprimitive theorem}, we show that $V$ is a witness to the linear dimension of $C_p\wr S_\ell$ over $\F_p$. Since  $V$ has at least $3$ composition factors, this shows that Corollary~\ref{cor:almost irred} fails for imprimitive groups.\hfill $\diamond$
\end{example}

 Any group $G$ that preserves a partition of a set $\Omega$ of size $k\ell  $ into $\ell$ parts each of size $ k$ is permutationally isomorphic to a subgroup of a wreath product $S_{ k}\wr S_{\ell}$. Further, if $\Delta \subseteq \Omega$ is a part of a partition $\Sigma$ preserved by $G$, setting $K=(G_\Delta)^\Delta$,  the permutation group induced by $G_\Delta$ on $\Delta$ and $L=G^\Sigma$, the permutation group induced by $G$ on $\Sigma$, we have that $G$ is permutationally isomorphic to a subgroup of $K\wr L$ (see \cite[Theorem 5.5]{PS}). Thus, by Lemma~\ref{lem: h leq g}, to find an upper bound on the linear dimension of an imprimitive group $G$, we study the linear dimension of imprimitive wreath products.
 
 Moreover, if we choose a partition $\Sigma$ with parts as small as possible and size larger than 1, then the group $K=(G_\Delta)^\Delta$ is in fact primitive. This allows us to use results already achieved for the primitive case to bound the linear dimension of imprimitive groups. Note that, in general, $G$ may preserve many partitions, see Example~\ref{eg:twowitnesses}. We now prove Theorem \ref{thm:imprimitive theorem}, which we restate for the  convenience of the reader.

\noindent\textbf{Theorem \ref{thm:imprimitive theorem}.}
\emph{
Let $K\leqslant \Sym(\Delta)$ be   primitive, let $L \leqslant \Sym(\Sigma)$, and let $G= K \wr L$ acting imprimitively on the set $\Omega=\Delta \times \Sigma$. Let $k=|\Delta|$ and $\ell=|\Sigma|$. Then
$$\lindim_F(G, \Omega) = \begin{cases}
     \ell+1 & \text{if }  $K$ \text{ is regular and }k =\mathrm{char}(F), \\
    \ell \lindim_F(K,\Delta) & \text{otherwise. }
\end{cases}$$}
\begin{proof}
 Let $W$ be an $F K$-module that is a witness to $\lindim_F(K,\Delta)$ and let $V$ be the formal direct sum $V=W\oplus W\oplus\cdots\oplus W$ of $\ell$ copies of $W$. Then we can make $V$ an $FG$-module by defining
$$(w_1,w_2,\ldots,w_\ell)^{(g_1,g_2,\ldots,g_\ell)}=(w_1^{g_1},w_2^{g_2},\ldots,w_\ell^{g_\ell})$$
and 
$$(w_1,w_2,\ldots,w_\ell)^\sigma=(w_{1\sigma^{-1}},w_{2\sigma^{-1}},\ldots, w_{\ell\sigma^{-1}})$$
for all $(w_1,w_2,\ldots, w_\ell)\in V$, $g_1,g_2,\ldots, g_\ell\in K$ and $\sigma\in L$.

Let $(\rho_1,\varphi_1)$ be the representation of the action of $K$ on $\Delta$ afforded by $W$. Write  
$\Omega=\Delta\times \{1,2,\ldots, \ell\}$ 
and define $\varphi:\Omega\rightarrow V$ so that $(i,j)\varphi$ is the vector for which all but the $j^{\mathrm{th}}$-coordinate is equal to $0$ and the $j^{\mathrm{th}}$-coordinate is equal to $(i)\varphi_1\in W$. Then letting $\rho:G\rightarrow \GL(V)$ be the induced homomorphism, we get that $(\rho,\varphi)$ is a representation for the action of $G$ on $\Omega$. Thus,
\begin{equation}
    \label{eq:imprimuprbnd}\lindim_F(K \wr L,\Omega) \leqslant \ell \lindim_F( K,\Delta).
\end{equation}

    Suppose first that  $K$ is not regular. 
    Since $K^\ell$ is an intransitive subgroup of $G$ with $\ell$ orbits of size $ k$, using Theorem~\ref{thm:intrans=} (see also Remark~\ref{rem: primitive implies pt stabs distinct}) and Lemma \ref{lem: h leq g}, we have 
    $$\ell \lindim_F(K,\Delta) = \lindim_F(K^\ell,\Omega) \leqslant \lindim_F( K\wr L,\Omega).$$
This gives the required lower bound, and together with the upper bound in \eqref{eq:imprimuprbnd}, establishes the theorem   when $K$ is not regular. 

We now assume that $K$ is regular, so that $K=C_k$ and $k$ is a prime number.  Write $G=P \rtimes L$ 
where $P = \langle x_1,\ldots,x_\ell\rangle \cong C_k^\ell$ and  $L
$ permutes the $x_i$ by acting on subscripts. 
Additionally, write $\Omega=\mathbb{Z}_k\times \{1,2,\ldots, \ell\}$ so that $x_i$ maps $(s,i)$ to $(s+1,i)$ (with operation modulo $k$ in the first coordinate) and fixes all other points. 
 We first  show that $\lindim_F(G,\Omega)\geqslant \ell$.

Let $U$ be a witness to $\lindim_F(G,\Omega)$, that  affords the representation $(\rho,\varphi)$, and  let $v_i:=(0,i)\varphi$ for each $i$. We show that $\{v_i\mid i\in\{1,2,\ldots, \ell\}\}$ is linearly independent.

Assume that $\sum_{i=1}^\ell a_iv_{i} = 0$. For each $s$, we have that
\[
    0=0^{x_s}=\left(\sum_{i=1}^\ell a_iv_{i}\right)^{x_s} =  a_s (1,s)\varphi+\sum_{i=1, i \neq s}^\ell a_i v_{i}
\]
and so taking the difference yields
 $$0  =  \sum_{i=1}^\ell a_i v_{i}-  a_s (1,s)\varphi-\sum_{i=1, i \neq s}^\ell a_i v_{i}  = a_s(v_{s}-(1,s)\varphi).$$ 
 Since $\varphi$ is injective,  it follows that $a_s=0$. Since we can do this with each $s\in\{1,2,\ldots,\ell\}$, we have that each $a_s=0$, and so $\{v_i \mid i\in\{1,2,\ldots, \ell\}\}$ is linearly independent. In particular, $\dim(U) \geqslant \ell$. 

  Suppose  that $k=\ch(F)$. We claim $(1,1)\varphi \notin \langle v_{i} \mid 1 \leqslant i \leqslant \ell \rangle$. Indeed, otherwise $(1,1)\varphi = \sum_{i=1}^\ell a_i v_{i}$ for some $a_i\in F$. Let $s>1$, and consider  $g = x_s \in P$. 
 We have $((1,1)\varphi)^g = (1,1)\varphi$, so $$0=(1,1)\varphi -((1,1)\varphi)^g = \sum_{i=1}^\ell a_i v_{i} - \left (\sum_{i=1}^\ell a_i v_{i} \right )^g = a_s(v_{s} -(1,s)\varphi )$$
and so $a_s=0$ for each $s>1$. Hence, $(1,1)\varphi  = a_1v_1$. Clearly $a_1\neq 0$, and applying $g=x_1\in P$ to $(1,1)\varphi  = a_1v_1$, we see that
\begin{align*}
    (2,1)\varphi  &= a_1(1,1)\varphi=a_1^2 v_1, \\
    (3,1)\varphi  &=a_1^3v_1,
\end{align*}
and so on. It follows that $v_1=(0,1)\varphi=a_1^{k}v_1$ so $a_1^k = 1$. 
Thus, $a_1$ is a root of the polynomial $f(t)=t^k-1 \in F[t]$. Since $\ch(F)=k$, we have $f(t)=(t-1)^k$, which means $a_1=1$ and so $(1,1)\varphi=v_1$,   a contradiction.
Hence, $\dim(U)\geqslant \ell+1$. From~\eqref{eq: cp wr sd d+1}, Lemma~\ref{lem: ext field} and Lemma~\ref{lem: h leq g} we have 
$$\lindim_F(K \wr L,\Omega) \leqslant \lindim_F(K \wr S_\ell,\Omega)\leqslant \lindim_k(K \wr S_\ell,\Omega)\leqslant \ell+1.$$
This  establishes the theorem  when $K$ is regular and $k =\ch(F)$.

We now suppose that  $k \neq \ch(F)$. Since both $\ch(F)$ and $k$ are prime, this also means $\ch(F) \nmid k= |K|$. For $1\leqslant i \leqslant \ell$, set  $V_i = \langle (x,i)\varphi|x\in \mathbb{Z}_k\rangle$ and note that $V_i$ is invariant under $\langle x_i\rangle$. By Maschke's Theorem applied to the nontrivial $F\langle x_i\rangle$-module $V_i$, we have $V_i = C_{V_i}(  x_i) \oplus U_i$   where~$U_{i}$ is a nontrivial $F\langle x_i\rangle$-module. For $i\neq j$, we see that $x_i$ centralises $V_j$, thus $V_i \cap \left (\sum_{j\neq i}V_j\right ) \leqslant C_{V_i}(  x_i  )$. Thus,
\[
\dim V \geqslant \dim\left (\sum_{i=1}^\ell V_i \right ) \geqslant \sum_{i=1}^\ell\dim \left ( V_i/C_{V_i}(  x_i  ) \right ) = \sum_{i=1}^\ell \dim(U_{i}).
\]
Now, let $1\leqslant i \leqslant \ell$  and set $X=U_{i }$ and $x=x_i$. We see that $X$ affords a representation for the group action of $\langle x \rangle \cong K$ over $F$. Indeed, since $X$ is a nontrivial module, there exists $u\in X$ with  $u^{x} \neq u$ and so the orbit of $u$ under $\langle x\rangle$ has length~$k$. 
 It follows that $\dim (X) \geqslant \lindim_F(K,\Delta)$ and so  $\dim(U)\geqslant \ell\lindim_F(K,\Delta)$. Together with~\eqref{eq:imprimuprbnd}, this completes the proof.  
\end{proof}

\begin{remark}
Lemma~\ref{lem: h leq g} can be used in combination with Theorem~\ref{thm:imprimitive theorem} to give upper bounds on groups acting imprimitively. These bounds are not necessarily  tight. For example, let $G=C_{15}$ viewed as a regular permutation group on $\Omega$ with cardinality  $15$ points. Then 
$$\lindim_2(C_{15},\Omega)=4$$
as witnessed by the embedding $G\cong \GL(1,2^4)$ in $\GL(4,2)$. On the other hand, $G$ acts imprimitively, with two block systems, featuring blocks of size $3$ and $5$. The above theorem would give the following bounds:
$$\lindim_2(C_{15},\Omega) \leqslant \lindim_2(S_3\wr S_5,\Omega )=5\cdot 2=10$$
and
$$\lindim_2(C_{15},\Omega ) \leqslant \lindim_2(S_5\wr S_3,\Omega)=3\cdot 4 = 12$$
which are far from accurate.
\end{remark}

The following example shows that the bound given by inclusion can be tight, and, gives us an example of a group with two distinct witnesses.

\begin{example}\label{eg:twowitnesses}
Let $G=C_6$ be the cyclic group of order $6$ and let $G$ act regularly on $\Omega$ so that~$\Omega$ has cardinality $6$. We have $\lindim_2(C_6,6)=4$ and there are two non-isomorphic witnesses. The witnesses arise from the two different imprimitive embeddings, $C_6 \hookrightarrow S_2 \wr S_3$ and $C_6 \hookrightarrow S_3 \wr S_2$. From Theorem~\ref{thm:imprimitive theorem} we have 
$$\lindim_2(S_2 \wr S_3,\Omega) = 3+1 = 4 \quad \text{and}\quad \lindim_2(S_3 \wr S_2, \Omega) = 2\cdot 2 =4$$
so that the witnesses $V_1$ and $V_2$ for $G_1=S_2\wr S_3$ and for $G_2=S_3 \wr S_2$, respectively are both witnesses for $G$.

Let $G=\langle x, y \rangle$ with $|x|=2$ and $|y|=3$. From Example~\ref{eg:cp wr sd}, we see that a witness $V_1$ for~$G_1$ has basis $\{e_1,e_2,e_3,e_4\}$ such that $\{e_1,e_2,e_3\}$ is an orbit for $y$ and $y$ fixes $e_4$. Thus, $C_{V_1}(y)=\langle e_1+e_2+e_3,e_4\rangle$ has dimension $2$. On the other hand, from the proof of Theorem~\ref{thm:imprimitive theorem} we see that the witness $V_2$ for $G_2$ is a direct sum of witnesses $U_1$ and $U_2$ for $S_3$ over $\F_2$ that are swapped by $S_2$. Since $U_1$ and $U_2$ are the deleted permutation module for $S_3$ over $\F_2$, we see that $C_{U_1}(y)=C_{U_2}(y)=0$. Since $V_2=U_1+U_2$, we have $C_{V_2}(y)=C_{U_1}(y)+C_{U_2}(y) = 0$. This shows that $V_1$ and $V_2$ are not isomorphic as $G$-modules.\hfill $\diamond$
\end{example}

\section{Wreath products in product action}\label{sec:prodaction}

An important class of primitive groups are wreath products in product action. Given  $K\leqslant \Sym(\Delta)$ and $L\leqslant S_\ell$, the wreath product $K\wr L=K^\ell\rtimes L$ acts on the Cartesian power $\Delta^\ell$ by
$$(\delta_1,\delta_2,\ldots, \delta_\ell)^{(h_1,h_2,\ldots, h_\ell)\sigma}= (\delta_{1\sigma^{-1}}^{h_{1\sigma^{-1}}},\delta_{2\sigma^{-1}}^{h_{2\sigma^{-1}}},\ldots, \delta_{\ell\sigma^{-1}}^{h_{\ell\sigma^{-1}}})$$
Note that $\sigma(h_1,h_2,\ldots,h_\ell)\sigma^{-1}=(h_{1\sigma},h_{2\sigma},\ldots, h_{\ell\sigma})$.
This action is primitive if and only if $K$ is primitive but not regular on $\Delta$, $\Delta$ is finite, and $L$ is transitive on $\{1,2,\ldots, \ell\}$ (see \cite[Theorem~5.18]{PS}).

\begin{lemma}
\label{lem: upper bound for product action}
Let $K\leqslant \Sym(\Delta)$ and $L\leqslant S_\ell$. Then 
$$\lindim_F(K\wr L,\Delta^\ell)\leqslant \ell\lindim_F(K,\Delta).$$
\end{lemma}
\begin{proof}
Let $V$ be a witness for $\lindim_F(K,\Delta)$ afforded by $(\rho,\varphi)$ and let $W$ be the direct sum of $\ell$ copies of $V$, and note $\dim(W)=\ell \lindim_F(K,\Delta)$. Define $\overline{\rho}:K\wr L\rightarrow \GL(W)$ by 
\begin{equation}
    \label{eq:action on V^ell}
(v_1,v_2,\ldots,v_\ell)^{((h_1,h_2,\ldots,h_\ell))\overline{\rho}}=(v_1^{(h_1)\rho},v_2^{(h_2)\rho},\ldots,v_\ell^{(h_\ell)\rho})
\end{equation}
and $$(v_1,v_2,\ldots,v_\ell)^{(\sigma)\overline{\rho}}=(v_{1\sigma^{-1}},v_{2\sigma^{-1}},\ldots, v_{\ell\sigma^{-1}})$$
for all $v_i\in V$, $h_i\in K$ and $\sigma\in L$. 

Pick $\delta \in \Delta$ and set $v=(\delta )\varphi$. Let $w=(v,v,\ldots,v) \in W$. We compute that
\[
    ((K \wr L)\overline{\rho})_w = (K_\delta \wr L)\overline{\rho}=((K \wr L)_{(\delta,\ldots,\delta)})\overline{\rho}.
\]
Thus, Lemma~\ref{lem: rep affords action} shows that $W$ affords a representation of the group action of $K \wr L$ on $\Delta^\ell$. This completes the proof.
\end{proof}

\begin{lemma}
\label{lem: upper bound for k^l when p div k}
Let $G = \Sym(\Delta) \wr S_\ell$ and suppose that $|\Delta|=k>2$. If $p$ is a prime with $p \mid k$,  then for any field $F$ of characteristic $p$, we have
$$\lindim_F(G,\Delta^\ell)\leqslant (k-2)\ell +1. $$
\end{lemma}

\begin{proof}
    As in the previous proof, let $V$ be a witness to $\lindim_F(\Sym(\Delta),\Delta)$, so that $\dim(V)=k-1$, and let $W$ be the direct sum of $\ell$ copies of $V$. Let $(\overline{\rho},\overline{\varphi})$ be the representation of the group action afforded by $W$. Fix $\delta \in \Delta$ and set $$u:= (\delta,\delta,\ldots,\delta)\overline{\varphi}.$$ Note that $(\delta,\ldots,\delta)$ is fixed  by the subgroup $S_\ell$ of $\Sym(\Delta)\wr S_\ell$, and hence $u \in W$ is fixed by the subgroup $(S_\ell)\overline{\rho}$.  
    By the proof of the previous lemma, $\Delta^\ell$ is in bijection with $u^{G\overline{\rho}}$. Let $U$ be the subspace of $W$ spanned by  $u^{G\overline{\rho}}=(\Delta^\ell)\overline{\varphi}$. Note that $(\overline{\rho}|_U,\overline{\varphi})$ is a representation of the group action, and so $\lindim_F(G,\Delta^\ell)\leqslant \dim(U)$.

    Since $p \mid k$, the permutation module $F^k$ for $\Sym(\Delta)$ has submodules $C$ and $C^\perp$ of dimensions $1$ and $k-1$ respectively with $C\subseteq C^\perp$. Moreover,  $V=F^k/C$,  and $V_0:=C^\perp/C$ satisfies $\dim(V_0)=k-2$.  In particular, $V/V_0$ has dimension $1$ and $\Sym(\Delta)$ acts trivially on $V/V_0$. Let $W_0$ denote the direct sum of $\ell$ copies of $V_0$, and  naturally identify $W_0$ as a submodule of $W$. Note that $\dim(W_0)=(k-2)\ell$, so $\dim(W/W_0)=\ell$, and $(\Sym(\Delta)^\ell)\overline{\rho}$ acts trivially on $W/W_0$. Hence, the vector $u+W_0 \in W/W_0$ is fixed by $(\Sym(\Delta)^\ell \rtimes S_\ell)\rho = G\rho$.
    
    In particular, $(W_0+U)/W_0 = \langle u^G + W_0\rangle = \langle u + W_0\rangle$ has dimension $1$. Thus,
\[
    \lindim_F(G,\Delta^\ell) \leqslant \dim( U) \leqslant \dim(U+W_0) = (k-2)\ell +1.\qedhere
\]
\end{proof}

\begin{lemma}
\label{lem: lower bound for k^l}
Let $K=\Sym(\Delta)$ and $L\leqslant S_\ell$  and suppose that $|\Delta|=k>2$.
Then
$$\lindim_F(K\wr L,\Delta^\ell)\geqslant (k-2)\ell +1. $$
\end{lemma}
\begin{proof} We write $\Delta=\{1,2,\ldots,k\}$.

Suppose $V$ is a witness to $\lindim_F(\Sym(\Delta) \wr L,\Delta^\ell)$ affording $(\rho,\varphi)$ and for $1\leqslant i \leqslant \ell$ and $2\leqslant j \leqslant k$,
set $v_{ij}=(1,\ldots,1,j,1,\ldots,1)\varphi$ where all entries are equal to $1$ except the $i$-th entry is equal $j$. Also set $v_1=(1,\ldots,1)\varphi$ (all entries equal to~$1$).

We claim that $X=\{v_1\} \cup \{v_{ij} \mid 1\leqslant i \leqslant \ell,3\leqslant j \leqslant k\}$ is linearly independent. Since $|X|=(k-2)\ell+1$, this proves the statement.

Consider a linear combination of 
$X$  equal to $0$, namely
\begin{equation}
\label{eq:lin comb2}    
\lambda_1 v_1+\sum_{\substack{1\leqslant i \leqslant \ell \\ 3\leqslant j \leqslant k}}  \lambda_{ij}v_{ij}=0.
\end{equation}

Pick $1\leqslant m \leqslant \ell$ and $3 \leqslant n \leqslant k$, and let $g_{mn}$ be the element of $K^\ell$, consisting of the transposition $(2,n)$ in the $m$-th entry and  the identity in all other entries. We calculate
\[
    (v_{ij})^{g_{mn}\rho}=\begin{cases} v_{ij} & \text{if } i \neq m, \\ v_{ij} & \text{if } i=m \text{ and } j\neq n, \\
v_{m2} &\text{if } i=m \text{ and }  j=n,\end{cases}
\]
and hence $g_{mn}\rho$ fixes all vectors in $X$, except for $v_{mn}$ which is mapped to $v_{m2}$.

Now we apply $g_{mn}\rho$ to \eqref{eq:lin comb2} above. Using linearity and the fact that $0$ is fixed by $g_{mn}\rho$, we get that 
\[
0 = \lambda_1 v_1+\sum_{\substack{1\leqslant i \leqslant \ell \\ 3\leqslant j \leqslant k\\i\neq m\\j\neq n}}  \lambda_{ij}v_{ij} + \lambda_{mn}v_{m2}
\]
and therefore
\[
    0=0-0^{g_{mn}\rho} = \lambda_{mn}(v_{mn}-v_{m2}).
\]
Since $\varphi$ is injective, we have that $v_{mn}-v_{m2}\neq 0$, and so $\lambda_{mn}=0$. Since $m$ and $n$ were arbitrary,  this argument shows $\lambda_{mn}=0$ for all    $1\leqslant m \leqslant \ell$ and $3\leqslant n \leqslant k$. Hence, \eqref{eq:lin comb2} yields $\lambda_1 v_1=0$ and thus $\lambda_1=0$.
This proves the claim.
\end{proof}

We now determine the linear dimension of the largest primitive wreath product that preserves the Cartesian product $\Delta^\ell$ where $|\Delta|=k$, namely,  $G=S_k \wr S_\ell$, and prove Theorem \ref{thm:prodaction}. Since any group acting primitively in product action is contained in (a conjugate of) $G$, Lemma~\ref{lem: h leq g} gives upper bounds for the linear dimension of such groups. Note that $S_2 \wr S_\ell$ acting on $2^\ell$ points is imprimitive since $S_2$ is regular, see \cite[Theorem 5.18]{PS}, so we assume $k\geqslant 3$.

We now prove Theorem \ref{thm:prodaction} and restate it for convenience of the reader.

\noindent\textbf{Theorem \ref{thm:prodaction}.}
\emph{Let $\Delta$ be a finite set of cardinality $k\geqslant 3$ and let $\ell\geqslant 2$. Suppose that $\ch(F)=p\geqslant 0$.
Then $$\lindim_F(S_k\wr S_\ell,\Delta^\ell)=
\begin{cases} 
(k-2)\ell+1 &\text{if }p\mid k, \\
(k-1)\ell & \text{if }p \nmid k.
\end{cases}$$}
\begin{proof}
If $p\mid k$,  the result follows immediately from Lemmas~\ref{lem: upper bound for k^l when p div k}  and \ref{lem: lower bound for k^l}. For the rest of the proof we assume that $p\nmid k$.

Let $V$ be the permutation module over $F$ for $S_k$ with basis $\{e_\delta\mid \delta\in \Delta\}$. Let 
$W=V\otimes V\otimes \cdots\otimes V$, the tensor product of $\ell$-copies of $V$. Then $W$ is an $F(S_k\wr S_\ell)$-module via the action
$$(e_{\delta_1}\otimes e_{\delta_2}\otimes\cdots\otimes e_{\delta_\ell})^{(h_1,h_2,\ldots,h_\ell)}=e_{\delta_1^{h_1}}\otimes e_{\delta_2^{h_2}}\otimes\cdots\otimes e_{\delta_\ell^{h_\ell}} $$
and
$$(e_{\delta_1}\otimes e_{\delta_2}\otimes\cdots\otimes e_{\delta_\ell})^\sigma= e_{\delta_{1\sigma^{-1}}}\otimes e_{\delta_{2\sigma^{-1}}}\otimes\cdots\otimes e_{\delta_{\ell\sigma^{-1}}}$$
Note that the basis $\{e_{\delta_1}\otimes e_{\delta_2}\otimes\cdots\otimes e_{\delta_\ell}\mid \delta_1,\ldots,\delta_\ell\in\Delta\}$ of $W$ is in bijection with $\Delta^\ell$ and so by Corollary \ref{cor:idpermmodule}, $W$ is isomorphic to the permutation module over $F$ for the action of $S_k\wr S_\ell$ on $\Delta^\ell$.  

Since $p\nmid k$, we have that $V=C\oplus C^\perp$. For each subset $S$ of $\{1,2,\ldots, \ell\}$, define 
$$W_S=\langle v_1\otimes v_2\otimes \cdots\otimes v_\ell \mid v_i\in C^\perp \textrm{ if } i\in S \textrm{ and } v_i\in C \textrm{ if } i\notin S\rangle.$$
So, for example, if $S=\{1,\ldots,m\}$ then $W_S = \underbrace{C^\perp \otimes \cdots \otimes C^\perp}_{m \textrm{ times}} \otimes \underbrace{C \otimes \cdots \otimes C}_{\ell -m \textrm{ times}}$. 
Note that $\dim(W_S)=(k-1)^{|S|}$ and by \cite[Lemma 4.4.3(vi)]{KL}, $W_S$ is irreducible as an $FS_k^\ell$-module. As a vector space, we may write $W=V \otimes W'$, where $W'=\underbrace{V\otimes V\otimes\cdots \otimes V}_{\ell-1 \textrm{ times}}$, and with any basis $\{x_j\}$  for $W'$ we have a direct sum decomposition 
\begin{equation}
    \label{eq:decomp}W=\bigoplus_j \left(V\otimes \langle x_j\rangle\right).
\end{equation}
We choose a basis $\{u_1,u_2,\ldots, u_k\}$ of $V$ so that $u_1,u_2\ldots, u_{k-1}\in C^\perp$ and $u_k\in C$. Let $\mathcal B = \{x_1,\ldots,x_{k^{\ell-1}}\}$ be the basis of $W'$  such that each vector in $\mathcal B$ is of the form $y_{j_2}\otimes y_{j_3}\otimes\cdots\otimes y_{j_{\ell}}$ with each $y_{j_i}\in \{u_1,u_2,\ldots,u_k\}$. With respect to the decomposition as in \eqref{eq:decomp}, we define $\pi_j:W\rightarrow V\otimes \langle x_j\rangle$ to be the projection map and for $S\subseteq\{1,2,\ldots,\ell\}$ define
 $X_S=\{x_j\mid (W_S)\pi_j\neq 0\}.$  Note that  $x_j\in X_S$ if and only if $x_j=y_{j_2}\otimes y_{j_3}\otimes\cdots\otimes y_{j_{\ell}}$ with $y_{j_i}= u_k$ precisely for $i\notin S$. Thus $X_{S}\cap X_{R}=\varnothing$ for all subsets $R,S$ of the same size with $R\neq S$.
Then as  $$W_S\subseteq \bigoplus_{x_j\in X_S} V\otimes \langle x_j\rangle$$
 it follows that $W_S\cap \langle W_R\mid |R|=|S|\textrm{ and } R\neq S\rangle =0$.
  
 For each $m$ such that $0\leqslant m\leqslant \ell$, let $W_m=\langle W_S\mid |S|=m\rangle$. Then $W_m=\bigoplus_{|S|=m}W_S$ and so  $\dim(W_m)=\binom{\ell}{m}(k-1)^m$.
The kernel of the action of $S_k^\ell$ on $W_{S}$ equals $\{(h_1,\ldots,h_\ell) \mid h_i =1$ for $i\in S\}$. Thus, for $R\neq S$, the kernels of the actions of $S_k^\ell$ on $W_S$ and $W_R$  are not equal. Hence the $W_S$ are pairwise nonisomorphic as $FS_k^\ell$-modules and so are the only irreducible $FS_k^\ell$-submodules of $W_m$.
As $S_\ell$ acts transitively on the set of $m$-subsets of $\{1,\ldots,\ell\}$, it follows that $W_m$ is an irreducible $F(S_k\wr S_\ell)$-submodule of $W$. Moreover,  as an $FS_k^\ell$-module, the composition factors of $W_m$   are  $W_S$ for $S\subseteq \{1,\ldots,\ell\}$ and  $|S|=m$. It follows  that 
$W_m$ and $\sum_{n\neq m} W_n$ have no common composition factor, and therefore $W_m \cap \sum_{n\neq m} W_n = 0$. Thus $\sum_{m=0}^\ell W_m = \bigoplus_{m=0}^\ell W_m$ and 
\[
    \dim(W)=k^\ell=(k-1+1)^\ell=\sum_{m=0}^\ell\binom{\ell}{m} (k-1)^m1^{\ell-m}=\sum_{m=0}^\ell \dim(W_m).
\]
Hence $W=\bigoplus_{m=0}^\ell W_m$ and the $W_m$ are the only  irreducible $F(S_k\wr S_\ell)$-submodules of $W$. Note that $\dim(W_1)=\ell(k-1)\leqslant
\dim(W_m)$ for all $m>1$ while $\dim(W_0)=1$. Thus, by Corollary \ref{cor:smallestcodim}, $\lindim_F(S_k\wr S_\ell,\Delta^\ell)=\ell(k-1)$.
\end{proof}

 \section{Actions on sets and partitions}\label{sec:Sn}

In this section, we focus on the symmetric and alternating groups $S_n$ and $A_n$ in some of their natural actions.
Let $\Omega=\{1,\ldots,n\}$ and consider the action of the symmetric group $S_n$ on $\Omega$. This action gives rise to an action on the set $\Omega_k$ of all $k$-subsets of $\Omega$. Note that if $\sigma \subseteq \Omega$ then $G_{\sigma} = G_{\overline{\sigma}}$, so the actions of $S_n$ on $\Omega_k$ and $\Omega_{n-k}$ are equivalent. For this reason, we may and do restrict attention to $1 \leqslant k \leqslant n/2$. Furthermore, if $n =2k$, then the stabiliser of an $\frac{n}{2}$-set is properly contained in the stabiliser of partition into two parts of size $\frac{n}{2}$ and so the actions of $S_n$ and $A_n$ on $\Omega_{\frac{n}{2}}$ are imprimitive. Thus, we only consider $1 \leqslant k < n/2$. Due to the exceptional nature of the permutation module for $S_n$ for small values of $n$, we will often only consider $n\geqslant 10$.

\begin{lemma}
\label{lem: upper bound for S_n on k sets}
Let $G=S_n$, let $F$ be a field and let $1\leqslant k < n/2$. Then $$   \lindim_F(G,\Omega_k)\leqslant n-1.$$
\end{lemma}
\begin{proof}
    Let $V=F^n$ be the permutation module for $G$ over $F$ with basis $\{e_1,\ldots,e_n\}$ and let $\rho$ be the representation of $G$ on $V$. Let $1\leqslant k < n/2$ and  let $\Omega_k$ be the set of $k$-subsets of $\{1,\ldots,n\}$, as defined above. We define $\varphi : \Omega_k \rightarrow V$ by 
    $$(\sigma)\varphi =\sum_{i\in \sigma} e_i \quad \text{for} \quad \sigma\in \Omega_k.$$
    It is easy to check that $(\rho,\varphi)$ is a representation of the action of $G$ on $\Omega_k$ over $F$.  
     This shows that $\lindim_F(G,\Omega_k) \leqslant n$.
     
     Let $C=\langle e_1+\cdots+e_n\rangle$ be the subspace of `constant' vectors of $V$. For $\sigma,\mu\in\Omega_k$ with $\sigma\neq \mu$, we see that $C+(\sigma)\varphi =C+(\mu)\varphi$ if and only  $n=2k$, $\ch(F)=2$  and $\mu$ is equal to the  complement of $\sigma$. Since $k<n/2$, we have that $\overline{\varphi}$ defined by $(\sigma)\overline{\varphi} = C + \sigma\varphi$ is an injection. Hence, with $\overline{\rho}$ as in \eqref{eq: barrho}, $(\overline{\rho},\overline{\varphi})$ is a representation of the action of $G$ on~$\Omega_k$. This gives   $
    \lindim_F(G,\Omega_k)\leqslant n-1$, as required.
\end{proof}

  We first consider the extreme values of $k$, which are an immediate consequence of Theorem \ref{thm:mods for Sn}. 
 \begin{corollary}
     Let $G=S_n$ or $A_n$ with $n\geqslant 5$. Then $\lindim_F(G,\Omega)=n-1$.
 \end{corollary}
 \begin{proof}
 This was shown in \cite[Proposition 22]{DAlconzo2024} for $S_n$ already, but we give another proof that also works for $A_n$. Note that by Lemma \ref{lem: h leq g} and  Lemma~\ref{lem: upper bound for S_n on k sets}, we have $\lindim_F(G,\Omega)\leqslant n -1$.
 Moreover, Lemma \ref{lem: h leq g} then also implies that if $\lindim_F(A_n,\Omega)=n-1$ then $\lindim_F(S_n,\Omega)=n-1$. Thus from now on we assume that $G=A_n$.
 
 Since $G$ is not cyclic, Lemma \ref{lem:cyclic}  yields $\lindim_F(G,\Omega)>1$. By Theorem~\ref{thm:fund thm}, we see that a witness to $\lindim_F(G,\Omega)$ is a quotient of the permutation module $F^n$.  From Theorem~\ref{thm:mods for Sn}, 
 the possible quotients of~$F^n$ are of dimensions $0,1,n-1$ and $n$. It follows that $\lindim_F(G,\Omega) \geqslant n-1$, and hence equality holds. 
 \end{proof}

Before dealing with $k\geqslant 2$, we need the following lemma.

 \begin{lemma}\label{lem: fix points of an}
     Suppose $F$ is a field, and $n$ and $k$ are integers with $1< k < n-1$. Let $G=A_n$ with $n\geqslant 5$, let $V=F^n$ be the permutation module for $G$ over $F$ and let $C$ be the subspace of constant vectors. If $\sigma \in \Omega_k$ and  $v+C \in V/C$ is fixed by $G_\sigma$, then  there are scalars $\lambda,\mu \in F$ such that $v=\lambda \sum_{i \in \sigma} e_i + \mu \sum_{i \notin \sigma } e_i$. 
 \end{lemma}
 \begin{proof}
       Suppose that $v+C\in V/C$ is fixed by $G_\sigma$. Then for all $g\in G_\sigma$ we have $v^g+C = v +C$, that is, $v-v^g \in C$.
       
       Write $v=\sum_{i=1}^n \alpha_i e_i$. For any $i, j\in \sigma$ with $i\neq j$, pick $i', j'\notin \sigma$ such that $i' \neq j'$. Then $g=(i,j)(i',j') \in G_\sigma$. Hence
       \[
        v-v^g= (\alpha_i-\alpha_j)(e_i -e_j) + (\alpha_{i'}-\alpha_{j'})(e_{i'}-e_{j'}) \in C.
        \]
        Since $n\geqslant 5$, this means $\alpha_i-\alpha_j =0=\alpha_{i'}-\alpha_{j'}$ and so $\alpha_i=\alpha_j$ and $\alpha_{i'}=\alpha_{j'}$. Thus, for all $i\in \sigma$ we have $\alpha_i=\lambda$ and for $i\notin \sigma$ we have $\alpha_i = \mu$, for some  scalars $\lambda,\mu\in F$ and the result follows.
 \end{proof}

\begin{lemma}
\label{lem:p divides n and k}
     Let $G=A_n$ or $S_n$ with $n\geqslant 10$ and let $F$ be a field of characteristic $p>0$.  Suppose that $1 < k<\frac{n}{2}$.
If $p\mid n$ and $p \mid k$, then $\lindim_F(G,\Omega_k)= n-2.$
\end{lemma}
\begin{proof}
Let $F^n$ denote the permutation module for $G$ acting on $n$ points and recall from Lemma~\ref{lem: upper bound for S_n on k sets} that $F^n$ affords a representation $(\rho,\varphi)$ where $\sigma\varphi=  \sum_{i\in \sigma} e_i$.  Since $p\mid k$, we have that $\sigma\varphi \in C^\perp$, and so $C^\perp$ also affords a representation for the action of $G$ on $\Omega_k$.

As seen in the proof of Lemma~\ref{lem: upper bound for S_n on k sets}, since $n\neq 2k$,  we have that $\sigma\varphi + C \neq \mu\varphi+C $ when $\sigma \neq \mu$. Thus, $C^\perp/C$ affords the induced representation $(\overline{\rho},\overline{\varphi})$ where $\sigma\overline{\varphi}=\sigma\varphi+C$. Hence, $\lindim_F(G,\Omega_k)\leqslant n-2$.  By Theorem~\ref{thm:mods for Sn}(3), we have $\lindim_F(A_n,\Omega_k)\geqslant n-2$, and so we are done for $G=A_n$ and then the result follows for $G=A_n$ by Lemma \ref{lem: h leq g}.
    \end{proof}

\begin{lemma}
\label{lem: n-2 no good for omega_k}
    Suppose $p$ is a prime, $F$ is a field of characteristic $p>0$, and $n$ and $k$ are integers with $1 < k < \frac{n}{2}$. If $p\mid n$ and $p\nmid k$, then the fully deleted permutation module $C^\perp/C$ is not a witness to $\lindim_F(A_n,\Omega_k)$. 
\end{lemma}
\begin{proof}
     Suppose that $C^\perp/ C$ affords the representation $(\rho,\varphi)$ of the action of $G:=A_n$ on $\Omega_k$  over $F$. Pick $\sigma \in \Omega_k$ and let $\sigma\varphi = v+C \in C^\perp/C$. Hence, $(G\rho)_{v+C}=(G_\sigma)\rho$  and by Lemma~\ref{lem: fix points of an}, we have $v=\lambda\sum_{i\in \sigma}e_i+\mu\sum_{i\notin \sigma}e_i$ for some $\lambda,\mu\in F$.      Since $v\in C^\perp$, we have that $k \lambda + (n-k)\mu =0$. As $p\mid n$, this means that $k\lambda -k\mu=0$. Since $p\nmid k$, we have that $k$ is invertible in $F$ and so $\lambda - \mu =0$. Hence $v=\lambda(e_1+\cdots + e_n) \in C$ and so $v+C = C$.  Thus $\sigma\varphi=C$ for all $\sigma$ and so $\varphi$ is not injective, a contradiction.
\end{proof}

We can now prove Theorem \ref{thm:Snksets} which we restate for convenience.

\noindent\textbf{Theorem \ref{thm:Snksets}.}
  \emph{  Let $G=A_n$ or $S_n$ with $n\geqslant 10$ and let $F$ be a field of characteristic~$p\geqslant 0$.  Let $1\leqslant k< \frac{n}{2}$. Then
    \[
    \lindim_F(G,\Omega_k) =    \begin{cases} n-2 &\textrm{ if }p\mid n \textrm{ and } p\mid k,  \\
  n-1 &\text{ otherwise.}\\
\end{cases}
    \]
} 
 \begin{proof}
 By Lemma~\ref{lem: h leq g}, we have $\lindim_F(A_n,\Omega_k)\leqslant \lindim_F(S_n,\Omega_k)$. Thus, in order to establish lower bounds we will consider only $A_n$, and for upper bounds we will  consider only $S_n$.
 
 First suppose that $p=0$ or $p\nmid n$. From Lemma~\ref{lem: upper bound for S_n on k sets}, we have $\lindim_F(S_n,\Omega_k) \leqslant n-1$. Since $p\nmid n$, we have $\lindim_F(A_n,\Omega_k)\geqslant n-1$ by Theorem~\ref{thm:mods for Sn}. Thus $\lindim_F(G,\Omega_k)=n-1$. 
 
 From now on, we  assume $p\mid n$.
 If $p \mid k$, then Lemma~\ref{lem:p divides n and k} gives  $\lindim_F(G,\Omega_k)=n-2$, as required.
Next suppose that $p\nmid k$. By Theorem \ref{thm:mods for Sn}(3) and Lemma \ref{lem:cyclic} we have that $\lindim_F(A_n,\Omega_k)\geqslant n-2$. Suppose that  equality holds and let $V$ be a witness. Then by Theorem~\ref{thm:mods for Sn}(3) again, $V$ is the fully deleted permutation module for $A_n$. However, this contradicts Lemma~\ref{lem: n-2 no good for omega_k} and so $\lindim_F(A_n,\Omega_k)\geqslant n-1$.
  Since $n\neq 2k$,  Lemma~\ref{lem: upper bound for S_n on k sets}  shows that $\lindim_F(S_n,\Omega_k)\leqslant n-1$ and so the result follows.
 \end{proof}

 Now we look at a case where the linear dimension of an action of $S_n$ is larger than $n$. Let $k,\ell \in \mathbb{N}$ with $k,\ell>1$ and $n=k\ell$. Denote by $U_k(n)$ the set of all partitions of $\{1,\ldots,n\}$ where every part is of size $k$, also called \emph{uniform partitions}. Then the action of $S_n$ on $U_k(n)$ is primitive.
 We will focus on the case where $F=\mathbb{C}$ and compute  $\lindim_\mathbb{C}(S_n,U_k(n))$. 
 Since the irreducible characters of the symmetric group are integer-valued, any $FS_n$-module with $F$ an algebraically closed field of characteristic $p$ can be realised over a field of size $p$ \cite[Proposition 2.10.8]{KL} and so our results can also be applied to fields of characteristic coprime to $|S_n|$.

  By Corollary \ref{cor:smallestcodim}, computing $\lindim_\mathbb{C}(S_n,U_k(n))$ comes down to computing the smallest dimension of an irreducible constituent of the permutation character that is bigger than 1. The permutation character is given by $\xi = 1 \uparrow_{S_{k}\wr S_{\ell}}^{S_n}$. Although we do not know the full decomposition, we know enough about the irreducible constituents to compute the linear dimension. First, we need the following lemma, which we presume is well-known; we include a proof here for the sake of completeness.

\begin{lemma}\label{lem:irreducible constituents permutation character}
    Let $G$ be a group and $K,H$ be subgroups of $G$ with $K \leqslant H$. If $\chi$ is an irreducible character of $G$, then
    \[
        \langle 1\uparrow_K^G,\chi \rangle = 0 \;\Longrightarrow\; \langle 1\uparrow_H^G,\chi \rangle = 0 .
    \]
    In other words, all irreducible constituents of the permutation character $1\uparrow_H^G$ are also irreducible constituents of the permutation character $1\uparrow_K^G$.
\end{lemma}

\begin{proof}
    By transitivity of induction, we have $1\uparrow_K^G = 1\uparrow_K^H\uparrow_H^G$. Write
    \[
        1\uparrow_K^H \;= \sum_{\psi \in \text{Irr}(H)} m_\psi \psi
    \]
    where the coefficients $m_\psi$ are non-negative integers. Then we find
    \[
        0= \langle 1\uparrow_K^G,\chi \rangle = \langle 1\uparrow_K^H\uparrow_H^G,\chi \rangle = \sum_{\psi \in \text{Irr}(H)} m_\psi \langle\psi\uparrow_H^G,\chi\rangle.
    \]
    Note that the summands on the right are non-negative. Hence, every summand is equal to 0. Now $\psi =1_H$ is an irreducible character of $H$ and the corresponding summand is given by
    \[
        m_{1_H}\langle 1\uparrow_H^G,\chi\rangle.
    \]
By Frobenius Reciprocity,  $m_{1_H} = \langle 1\uparrow_K^H,1_H \rangle = \langle 1_K,1_K \rangle = 1$, and so we obtain $\langle 1\uparrow_H^G,\chi\rangle = 0$ as desired.
\end{proof}

  The irreducible characters of $S_n$ are labelled by  partitions $\lambda$ of $n$ and we denote them by $\chi^\lambda$.

 \begin{theorem}\label{thm:uniform partitions}
     If $n \geqslant 9$, then we have
     \[
        \lindim_\mathbb{C}(S_n,U_k(n)) = \frac{n^2-3n}{2}.
     \]
 \end{theorem}

 \begin{proof}
     In Chapter 12.9 of \cite{GodsilMeagher}, the following statements are shown.
     \begin{enumerate}[(1)]
         \item
            Let $k,\ell > 1$ and $n=k\ell$. The permutation character $\xi = 1 \uparrow_{S_{k}\wr S_{\ell}}^{S_{n}}$ contains the irreducible character $\chi^{(n-2,2)}$ as an irreducible constituent but not the irreducible character $\chi^{(n-1,1)}$.
         
         \item
            Let $n \geqslant 9$ and $\chi^{\lambda}$ be an irreducible character of $S_n$ of dimension less than $(n^2-3n)/2$. Then $\lambda$ is one the following partitions:
            \[
                (n),(1,\ldots,1),(n-1,1),(2,1,\ldots,1).
            \]
     \end{enumerate}
     Using the hook length formula (see \cite[Theorem 3.10.2]{sagan2001}), we find that
     \[
        \dim \chi^{(n-2,2)} = \frac{n^2-3n}{2}.
     \]
     So from (1), we know that $\lindim_\mathbb{C}(S_n,U_k(n)) \leqslant \dim \chi^{(n-2,2)} = (n^2-3n)/2$.
     
     Now we show that the linear dimension cannot be less than $(n^2-3n)/2$. By (2), there are only four irreducible characters of dimension less than $(n^2-3n)/2$. The character $\chi^{(n)}$ is the trivial character and it has dimension 1, so it does not give the linear dimension. The character $\chi^{(n-1,1)}$ is not a constituent of $\xi$ by (1). It remains to show that neither $\chi^{(1,\ldots,1)}$ nor $\chi^{(2,1,\ldots,1)}$ are   irreducible constituents of $\xi$.
     
     We have that $S_k \times \cdots \times S_k$ is a subgroup of $S_k \wr S_\ell$. Hence, by Lemma \ref{lem:irreducible constituents permutation character} every irreducible constituent of $\xi = 1 \uparrow_{S_{k}\wr S_{\ell}}^{S_n}$ is an irreducible constituent of $1 \uparrow_{S_{k} \times \cdots \times S_{k}}^{S_n}$. 
     The decomposition of the latter is known by Young's rule (see \cite[Theorem 2.11.2]{sagan2001}). We have that $\chi^{\lambda}$ is an irreducible constituent of $1 \uparrow_{S_{k} \times \cdots \times S_{k}}^{S_n}$ if and only if $\lambda \unrhd (k,\ldots,k)$ where $\unrhd$ is the dominance order for partitions (see \cite[Section~2.2]{sagan2001} for details). Since $k>1$, we find that $(1,\ldots,1),(2,1,\ldots,1) \not\hspace*{-0.7ex}\unrhd\, (k,\ldots,k)$. Thus,  $\chi^{(1,\ldots,1)}$ and $\chi^{(2,1,\ldots,1)}$ are not irreducible constituents of $1 \uparrow_{S_{k} \times \cdots \times S_{k}}^{S_n}$, and hence not of $\xi$.
     
     So the smallest dimension of an irreducible constituent of $\xi$ that is larger than 1 is $(n^2-3n)/2$ and we obtain $\lindim_\mathbb{C}(S_n,U_k(n)) = (n^2-3n)/2$.
   \end{proof}

\section{Finite 2-transitive groups}

In this section, we prove Theorem \ref{thm:2transLindim}, which determines the linear dimension of 
the finite almost simple 2-transitive groups. We do this in two lemmas. Before presenting these lemmas, we give upper bounds for the linear dimension of the affine groups in natural characteristic, which may be applied to any finite 2-transitive affine group.
See \cite[Sections~7.3 and~7.4]{cameron1999permutation} for the classification of finite 2-transitive groups.

\begin{example}
\label{eg:agldp}
Let $d$ be an integer and $F$ a field. Let $V=F^{d+1}$ considered as row vectors with standard basis $\{e_1,\ldots,e_{d+1}\}$ and let 
$$G = \left \{ \left ( \begin{array}{cc}A & \bf{0} \\ \underline{v} & 1 \end{array} \right ) \mid \underline{v}\in F^d,\ A\in \GL(d,F) \right \} \leqslant \GL(d+1,F)$$
where we view $\bf{0}$ as the $d\times 1$ zero column vector and $\underline{v}\in F^d$ as a $1 \times d$ row vector. Note that $G\cong \AGL(d,F)$ and  $G$ fixes the hyperplane $\langle e_1,\ldots,e_d\rangle$. By direct calculation, we see that $G_{e_{d+1}}\cong \GL(d,F)$. Hence, $(e_{d+1})^G$ is in bijection with $F^d$ and the action of $G$ on this orbit is equivalent to the action of $\AGL(d,F)$ on $F^d$. By Lemma~\ref{lem: rep affords action}, we see
$$\lindim_F(\AGL(d,F),F^d) \leqslant d+1.$$

Suppose  that $|F|=q=p^e$ for some prime $p$ and integer $e$. Since $|\AGL(d,q)|>|\GL(d',q)|$ for any integer $d'\leqslant d$, it is clear that $\lindim_q(\AGL(d,q),\F_q^d)\geqslant d+1$. Hence, using the upper bound above, $\lindim_q(\AGL(d,q),\F_q^d) = d+1$.

Suppose now that $e=rs$ for some integers $r$, $s$, and write $q_0=p^s$. We have a natural embedding $\GL(d,q_0^r) \rightarrow \GL(dr,q_0)$ and an identification $\F_{q_0^r}^d\rightarrow \F_{q_0}^{dr}$. In this way, we see $\AGL(d,q_0^r)$ as a subgroup of $\GL(dr+1,q_0)$. Thus
$$\lindim_{{p^{s}}}(\AGL(d,\F_{p^e}),\F_{p^e}^d) \leqslant dr+1.$$
Now suppose that $t$ is an integer. Then with $s=\gcd(e,t)$ and $r=e/s$, we have that $\F_{p^s}$ is the smallest common subfield of $\F_{p^e}$ and $\F_{p^t}$ and Lemma~\ref{lem: ext field} gives 
$$\lindim_{{p^{t}}}(\AGL(d,\F_{p^e}),\F_{p^e}^d) \leqslant dr+1.$$
By Lemma~\ref{lem: h leq g}, the above inequalities give bounds for all affine groups. \hfill $\diamond$
\end{example}

\begin{lemma}\label{lem:2transNoBTs}
 Let $F$ be a field of characteristic $p\geqslant 0$, and let $G$ be a $2$-transitive group of degree $n$ acting on $\Omega$, where $\soc(G)\neq\Suz(q)$ or $\Ree(q)$. Then either $\lindim_F(G,\Omega)=n-1$, or $\lindim_F(G,\Omega)$ is as in Table~\ref{tab:2transLinDim}.
\end{lemma}

\begin{proof}
 Since every $2$-transitive group is primitive and not cyclic, Corollary~\ref{cor:quotientPermMod} allows us to find the linear dimension as the codimension of the largest non-hyperplane submodule of the permutation module. As in Section~\ref{sec:witness}, the fact that the permutation module is self-dual means that it suffices to find a submodule $U$ of $F^\Omega$ having minimal dimension such that $\dim(U)>1$. Given a $2$-transitive group $G$, \cite{mortimer1980modular} studies $F^\Omega$ via the irreducibility of the fully deleted permutation module $H=C^\perp/(C\cap C^\perp)$, called the \emph{heart} in \cite{mortimer1980modular}, where $C=\langle e_1+\cdots+e_n\rangle$. First of all, if $p=0$, then $H$ is irreducible, and the result holds. If $p>0$, then conditions regarding the reducibility of $H$ are given in \cite[Table~1]{mortimer1980modular}; in the remainder of the proof we work through the lines of this table where $H$ may be reducible, except for the cases $\soc(G)=\Suz(q)$ or $\Ree(q)$, which are treated in Lemma~\ref{lem:2transNoBTs}.

 Suppose that $\soc(G)=\PSL_d(q)$, $d\geqslant 3$, $q=p^t$, and $\Omega$ is the set of points of $\PG_{d-1}(q)$. We wish to show that $\lindim_{\F_p}(G,\Omega)=\ell$, where $\ell=\binom{d+p-2}{d-1}^t$. First, \cite[Theorem~5.7.1]{assmus1992designs} shows that $\F_p^\Omega$ contains a submodule of dimension $\ell+1$, namely the design submodule $D$ generated by the characteristic vectors of the hyperplanes of $\PG_{d-1}(q)$. Since $D$ contains the constant vectors $C$, and here $p\nmid n$, we have that $D\cap C^\perp$ is a submodule of $\F_p^\Omega$ of dimension $\ell$. 
 It remains for us to show that $D\cap C^\perp$ is the unique minimal submodule and does not split over a larger field, for which we apply the results of \cite{bardoe2000permutation}.
 
 First, \cite[Theorem~A(a)]{bardoe2000permutation} shows that, working over an algebraically closed extension $k$ of $F$, the composition factors of $k^\Omega$ are $C\otimes k$ and a set of simple modules $L(s_0,\ldots,s_{t-1})$ (in the notation of \cite{bardoe2000permutation}), where the parameters $(s_0,\ldots,s_{t-1})$ come from an admissible set $\mathscr{H}$ defined in the hypotheses of \cite[Theorem~A]{bardoe2000permutation}. Note that by \cite[Eq.~(15)]{bardoe2000permutation} and the surrounding discussion, the modules $L(s_0,\ldots,s_{t-1})$ are simple $\SL_d(q)$-modules. By \cite[Theorem~A(c) and (d)]{bardoe2000permutation} and the fact that $(1,\ldots,1)$ is the unique minimal element in the partially ordered set $(\mathscr{H},\leq)$, every non-trivial submodule of $C^\perp\otimes k$ has the submodule $L(1,\ldots,1)$ as a composition factor. By \cite[Corollary~2.1]{bardoe2000permutation}, $L(1,\ldots,1)$ has dimension $\binom{d+p-2}{d-1}^t$. Finally, \cite[Theorem~B]{bardoe2000permutation} and \cite[Section~8]{bardoe2000permutation} show that $D\cap C^\perp$ is a subspace of $L(1,\ldots,1)\oplus (C\otimes k)$, and hence is the required $F\, \PSL_d(q)$-module. Thus line 1 holds.

 For ${\rm A}_7$ acting on $15$ points, and the two actions of $\Sp_{2d}(2)$, we need only consider the case $p=2$. Applying the results of \cite{ivanov1993finite}, line 2 of Table~\ref{tab:2transLinDim} is given by \cite[Theorem~5.1]{ivanov1993finite} and observing that, by \cite[Proposition~5.3.7]{KL}, the minimal degree of a representation of ${\rm A}_7$ in characteristic $2$ is $4$. Furthermore, lines 3 and 4 are given by \cite[Theorem~6.2]{ivanov1993finite}, and the fact that, by \cite[Proposition~5.4.13]{KL}, the minimal degree of a representation of $\Sp_{2d}(2)$ in characteristic $2$ is given by the natural module of dimension $2d$. Note that the isomorphisms $\Sp_4(2)'\cong{\rm A}_6\cong\PSL_2(9)$ mean that the actions of $\Sp_4(2)'$ on $6$ and $10$ points are covered under the treatment of ${\rm A}_6$ and $\PSL_2(9)$, respectively.
 
 The case where $G$ acts $2$-transitively on $q+1$ points and $\soc(G)=\PSL_2(q)$ is treated in \cite[Section~3(F)]{mortimer1980modular}. If $G$ is $2$-transitive but not $3$-transitive, then $H$ is reducible when $p=2$ and $q\equiv\pm 1{\pmod 8}$, or when $F$ contains $\F_4$ and $q\equiv\pm 3{\pmod 8}$, giving lines~5 and~6 of Table~\ref{tab:2transLinDim}. 

 If $\soc(G)=\PSU_3(q)$, with $G$ acting on $q^3+1$ points, then \cite[Section~G]{mortimer1980modular} shows that $H$ is reducible if and only if $p$ divides $q+1$. Line 8 then follows from \cite[Theorem~4.1]{hiss2004hermitian}.

 Finally, we consider the sporadic groups for which $H$ is reducible. Lines 12, 13, 14, 17 and 19 are given by \cite[Section~8]{ivanov1993finite}. Lines 15 and 16 are given by \cite[Section~3(J)]{mortimer1980modular}. Lines 18 and 20 were confirmed by calculations in \textsc{Magma}.
\end{proof}

\begin{lemma}\label{lem:2transBTs}
 Let $F$ be a field of characteristic $p\geqslant 0$, and let $G$ be a $2$-transitive group of degree $n$ acting on $\Omega$, where $\soc(G)=\Suz(q)$ or $\Ree(q)$. Then either $\lindim_F(G,\Omega)=n-1$, or $\lindim_F(G,\Omega)$ is as in Table~\ref{tab:2transLinDim}.
\end{lemma}

\begin{proof}
 We proceed in a similar manner to the proof of Lemma~\ref{lem:2transNoBTs}, by analysing the structure of the heart $H$ of the permutation module $F^\Omega$. However, for this proof we appeal to the theory of Brauer trees for cyclic defect groups; see, for example, \cite[Section~17]{alperin1993local}. The vertices of a Brauer tree for a group $G$ are labelled by some irreducible ordinary characters of $G$, and the edges are labelled either by irreducible Brauer characters or their corresponding simple modules, depending on the application. Labelling edges by Brauer characters allows one to determine relations among Brauer characters while, more importantly for us, the structure of a \emph{planar} Brauer tree and the simple modules labelling the edges allows one to determine the structure of certain modules. 
 In particular, given a planar Brauer tree that is a star with an edge labelled by a simple module $S$, the projective lift $P_S$ of $S$ is uniserial  
 (that is, the submodules of $P_S$ form a totally ordered chain) with both the first and last composition factors of $P_S$ being isomorphic to $S$, and with the remaining composition factors, and the order in which they appear, being determined by the ordering of the relevant edges around the central vertex of the Brauer tree.
 Note that in the remainder of the proof we assume that $G$ is simple. As we will see, in each case $F^\Omega$ turns out to be uniserial, and hence any automorphism of $G$ acts trivially on the submodule lattice, and thus the result continues to hold in the case that $G$ is almost simple.

 Let $G=\Suz(q)$ with $q=2^{2d+1}$, and let $m^2=2q$. Since $q^2+1=(q+m+1)(q-m+1)$, with the two factors here being coprime, this gives the two cases $p\mid (q+m+1)$ and $p\mid (q-m+1)$. In \cite[Section~3(E)]{mortimer1980modular} it is shown that $H$ is irreducible in the case where $p\mid (q-m+1)$, so from now on we assume that $p\mid (q+m+1)$. In this case, \cite[Case~1, p.11]{mortimer1980modular} shows that $H$ is reducible, though there is a small mistake in the analysis on page 12 of \cite{mortimer1980modular}. In particular, the phrase `$\U\leqslant \M$ or $\U^\perp\leqslant M$' should read `$\U\leqslant \M^\perp$ or $\U\geqslant \M$', and there is also a case missing from the analysis in the last paragraph of Section~3(E). To complete the analysis, we apply the theory of Brauer trees for cyclic defect groups (see, for example, \cite[Section~17]{alperin1993local}). The Brauer tree for this case is correctly given in \cite[Case~1, p.11]{mortimer1980modular} (with vertices labelled by the characters $1$, $\omega_1$, $\omega_2$ and $\eta_j$), and it is also pointed out that the defect group is cyclic. The fact that $p$ divides $q+m+1$ implies that $p$ does not divide the order of the point-stabiliser of $G$,
 which consequently means that $F^\Omega$ is a projective module (see the first paragraph of the proof of \cite[Lemma 3]{Neumann1972}). Indeed, since $F^\Omega$ is indecomposable with one-dimensional head, $F^\Omega$ is the projective lift of the trivial module. Following \cite[Section~17]{alperin1993local}, we can deduce the structure of $F^\Omega$ as the projective lift of the edge adjacent to the character $1$ in the Brauer tree (\cite[Case~1, p.~11]{mortimer1980modular}). Since the Brauer tree is a star, we immediately see that $F^\Omega$ is uniserial. Moreover, the degrees $m(q-1)/2$, $(q-1)(q-m+1)$ and $m(q-1)/2$ of $\omega_1$, $\eta_j$ and $\omega_2$, which sum to $q^2-1=\dim(H)$, correspond to the dimensions of the composition factors of $H$, in ascending order (corresponding to a clockwise rotation around the vertex $\alpha$ in the Brauer tree). Thus, $\lindim_p(G,\Omega)=m(q-1)/2+1$, and line 7 of Table~\ref{tab:2transLinDim} holds.

 Suppose that $G=\Ree(q)$ and $G$ acts on $q^3+1$ points, with $q=3^{2d+1}$, and let $m^2=3q$. First, if $p=2$ then \cite[Prop.~3.8 and Thm.~3.9]{landrock1980principal} gives that the linear dimension is $q^2-q+1$, so that line~9 holds. Assuming now that $p\neq 2$, we
 need to consider, by \cite[Section~3(H)]{mortimer1980modular},
 the cases where $p$ divides either $q+1$, $q+m+1$ or $q-m+1$.
 Note that as both $m$ and $q$ are non-trivial powers of $3$, we have that $3$ does not divide any of the factors $q+1$, $q+m+1$ or $q-m+1$.
 Thus $p\neq2,3$ and, by \cite[Section~2.1]{hiss1991brauer}, the defect group is indeed cyclic in each case, so that the Brauer trees are given in \cite[Theorems~4.2--4.4]{hiss1991brauer}. Since $p\neq2,3$ and $p$ divides $q^3+1$, it follows that $p$ does not divide the order of the point-stabiliser of $G$. Thus $F^\Omega$ is projective and, in particular, is the projective lift of the trivial module. Thus, we may find the structure of $F^\Omega$ by analysing the edge adjacent to the trivial character $\xi_1$ in the relevant Brauer tree. First, if $p\mid q+1$ then we need to investigate the tree in \cite[Theorem~4.2(a)]{hiss1991brauer}. The permutation module for this case is investigated more closely in \cite[p.~59]{hiss1997incidence}. In particular, it is shown that $F^\Omega$ is uniserial and $H$ has composition factors of dimensions $(mq^2-m)/6+(q^2-q)/2$, $(q-1)(q^2-mq/3-m/3+1)$ and $(mq^2-m)/6+(q^2-q)/2$, in ascending order. Thus line~10 holds. If $q$ divides $q+m+1$ then the relevant planar embedded Brauer tree is given in \cite[Theorem~4.3]{hiss1991brauer} (note that there is a typo in \cite[Theorem~4.3]{hiss1991brauer}; the character $\xi_7$ should instead be $\xi_6$, see the proof of Theorem~4.3 on pages 882--883 of \cite{hiss1991brauer}). Since the Brauer tree is a star, it follows that $F^\Omega$ is uniserial. 
 As $F^\Omega$ is self-dual, we may assume that $S_1,S_2,S_3,S_2^*,S_1^*$ are the composition factors of $H$, in ascending order, with Brauer characters given by the restriction to $p$-regular classes of $\xi_4,\xi_5,\chi_t,\xi_6,\xi_3$, respectively, where $S_1^*,S_2^*$ are the duals of $S_1,S_2$ and $S_3$ is self-dual. Thus,
 \begin{align*}
  \dim(S_1)&=\dim(S_1^*)\leqslant \frac{m}{3}(q^2-1),\\
  \dim(S_2)&=\dim(S_2^*)\leqslant \frac{m}{6}(q-1)(q-m+1),\,\,\text{and},\\
  \dim(S_3)&\leqslant (q^2-1)(q-m+1).
 \end{align*}
 The fact that $\dim(H)=q^3-1$ implies that the three inequalities above are in fact equalities.
 Thus the linear dimension is $m(q^2-1)/3+1$, proving line 11. For the remaining case, where $p$ divides $q-m+1$, the Brauer tree of interest is given in \cite[Theorem~4.3]{hiss1991brauer}, and this shows that $H$ is irreducible. This completes the proof.
\end{proof}
 
\section{Some open problems}\label{sec:problems}

We end the paper with  open problems and further directions of research.

Our results have highlighted the importance of determining the linear dimensions of primitive permutation groups and in particular those that are not contained in a wreath product in product action. The remaining primitive permutation groups fall into three classes: affine, diagonal type and almost simple. We determined $\lindim_F(\AGL(d,F),F^d)$ in Example \ref{eg:agldp} and looked at some almost simple groups in Section \ref{sec:Sn} and Theorem \ref{thm:2transLindim}. Other examples of primitive almost simple groups are the symmetric and alternating groups in their `standard' actions: on partitions and on subsets. The linear dimension of the latter action is determined in Theorem~\ref{thm:Snksets}. This suggests the following natural problems.

\begin{problem}
 Determine the linear dimension of primitive groups $G<\AGL(d,q)$ for all fields.
\end{problem}

\begin{problem}
   Determine the linear dimension of diagonal type groups. 
\end{problem}

\begin{problem}
    Determine the linear dimension of primitive actions of almost simple groups, for example for the action of $\PSL_n(q)$ on the set of $k$-dimensional subspaces of an $n$-dimensional vector space with $1<k<n-1$.
\end{problem}

We saw in Theorem \ref{thm:Snksets} that for $n\geqslant 10$, the linear dimension of $S_n$ acting on $k$-sets is either $n-1$ or $n-2$. We suspect that it  holds for $n\geqslant 5$. We also saw in Theorem \ref{thm:uniform partitions} that for the action on uniform partitions the linear dimension over $\mathbb{C}$ is $(n^2-3n)/2$. This suggests the natural question.

\begin{problem}
Let $S_n$ act primitively on a set $\Omega$ such that $\Omega$ is not the set of $k$-subsets of $\{1,2\ldots,n\}$. Is it the case that $\lindim_F(S_n,\Omega)\geqslant (n^2-3n)/2$?
\end{problem}

For primitive groups $G$, we saw in Corollary~\ref{cor:almost irred} that a witness is either irreducible, or, has a unique codimension one submodule, which is irreducible.  We are not aware of any conditions on either the group $G$ or the field $F$ to guarantee that a witness is irreducible. However, if a witness is not irreducible, this has implications for $1$-cohomology:  there exists a module $V$ for $G$ such that $H^1(G,V)\neq 0$, see \cite[(17.11)]{Aschbacher}. This raises:
\begin{problem}
What are necessary and sufficient conditions for a witness to be irreducible?    
\end{problem}

For imprimitive groups $G \leqslant S_k \wr S_\ell$, we saw in Example~\ref{eg:cp wr sd} that a witness can have more than two composition factors. In the example given, the composition factors are all composition factors of the permutation module for $S_\ell$ on $\ell$ points -- and are thus restricted. We are therefore motivated to pose the following:
\begin{problem}
    Is there an absolute bound on the number of composition factors of a witness~$V$ to $\lindim_F(G,\Omega)$ for transitive groups $G$?
\end{problem}

There is no absolute bound on the number of composition factors of witnesses  for intransitive groups. Let $N\geqslant 3$ be a natural number.  For each $k$, let $S_k$ act naturally on the set $I_k$ of size $k$.
Then the direct product $G=\prod_{k=3}^{N+2} S_k $ acts intransitively on the disjoint union   $\Omega = \dot\bigcup_{k=3}^{N+2}I_k$ of size $\sum_{k=3}^{N+2} k$. By Theorem \ref{thm:intrans=}, a witness $V$ for $\lindim_F(G,\Omega)$  must be the direct sum $V=\bigoplus_{k=3}^{N+2} V_k$ where each $V_k$ is a witness to $\lindim_F(S_k,\Omega_k)$. In particular, we see that $V$ has at least $N$ composition factors.

If $F$ is an extension of the field $L$ then Lemma \ref{lem: ext field} gives bounds for $\lindim_L(G,\Omega)$ in terms of $\lindim_F(G,\Omega)$.  It is desirable to understand this relationship.

\begin{problem}
    If $F$ is an extension of the field $L$, give an exact formula for $\lindim_F(G,\Omega)$ in terms of $\lindim_L(G,\Omega)$.
\end{problem}

In \cite[Proposition 23]{DAlconzo2024} it is shown that there is a homomorphism $\rho:S_{2^n} \rightarrow \AGL(2^n-2,2)$ and an injection $\iota:\{1,2,\ldots,2^n\} \rightarrow \F_2^{2^n-2}$ such that \begin{equation}\label{intertwin2}
(\omega^g)\iota= \left(\omega\iota\right)^{g\rho}.
\end{equation} 
One could define the \emph{affine dimension} of a group action $G$ on $\Omega$ over a field $F$ to be the smallest integer $d$ such that there is a homomorphism $\rho:G\rightarrow \AGL(d,F)$ and injection $\iota:\Omega\rightarrow F^d$ such that (\ref{intertwin2}) holds.

\begin{problem}
 Develop a theory for the affine dimension of a group action.   
\end{problem}

A similar possibility is to develop a theory for the \emph{projective dimension} of a group action. This would be the smallest $d$ such that there is a homomorphism $\rho:G\rightarrow \PGL_d(F)$ and an injection $\iota$ from $\Omega$ to the set of 1-dimensional subspaces of $F^d$ such that 
(\ref{intertwin2}) holds.


\end{document}